\documentclass[12pt,leqno,a4paper] {amsart}
\usepackage{amssymb,enumerate, verbatim}
\newtheorem{propo}{Proposition}[section]
\newtheorem{defi}[propo]{Definition}

\newtheorem{lemma}[propo]{Lemma}
\newtheorem{corol}[propo]{Corollary}

\newtheorem{theo}[propo]{Theorem}

\newtheorem{rem}[propo]{Remark}

\newtheorem{prob}[propo]{Problem}

\theoremstyle{definition}
\newtheorem{defn}[propo]{Definition}
\newtheorem{remk}[propo]{Remark}
\newcommand{\ld}{,\ldots ,}

\newcommand{\ra}{ \rightarrow }

\newcommand{\lan}{ \langle }
\newcommand{\ran}{ \rangle }

\newcommand{\diag}{\mathop{\rm diag}\nolimits}

\newcommand{\Id}{\mathop{\rm Id}\nolimits}
\newcommand{\Irr}{\mathop{\rm Irr}\nolimits}

\newcommand{\Om}{\Omega}

\newcommand{\ZZ}{\mathop{\Bbb Z}\nolimits}

\newcommand{\al}{\alpha}

\newcommand{\ep}{\varepsilon}

\newcommand{\lam}{\lambda }

\newcommand{\up}{^{-1}}

\newcommand{\si}{\sigma }
\newcommand{\Lie}{{\rm {Lie}}(G)}
\def\11{$(1)$}
\def\22{$(2)$}
\def\33{$(3)$}

\def\d12{{_{12}}}

\def\ei{{eigenvalue }}
\def\eis{{eigenvalues }}

\def\f{{following }}

\def\ho{{homomorphism }}

\def\ii{{if and only if }}

\def\ir{{irreducible }}

\def\irr{{irreducible representation }}

\def\irrs{{irreducible representations }}
\def\itf{{It follows that }}

\def\mult{{multiplicity }}

\def\rep{{representation }}
\def\reps{{representations }}

\def\St{{Suppose that }}

\def\SL{{\rm SL}}
\def\GL{{\rm GL}}
\def\Sp{{\rm Sp}}

\newcommand{\el}{\end{lemma}}
\newcommand{\om}{\omega }

\newcommand{\bl}{\begin{lemma}\label}

\begin{document}

\title[Spectra of non-regular elements]{Spectra of non-regular elements in irreducible representations of simple algebraic groups}
\date{\today}
\author{Donna M. Testerman}
\address{Institut de Math\'ematiques, Station 8, \'Ecole Polytechnique
  F\'ed\'erale de Lausanne, CH-1015 Lausanne, Switzerland.}
\email{donna.testerman@epfl.ch}
\author{Alexander Zalesski}
\address{Department of Physics,
Mathematics and Informatics,
Academy of Sciences of Belarus,
66 Prospect Nezalejnasti,
Minsk 220000, Belarus}
\email{alexandre.zalesski@gmail.com}

\keywords{semisimple elements, irreducible representations, eigenvalue multiplicities, simple linear algebraic groups}

\thanks{Testerman was supported by the Fonds National Suisse de la Recherche
Scientifique grant number 200021-175571. The second author acknowledges the hospitality of the EPFL Institute of Mathematics during research visits to Lausanne.}

\subjclass[2010]{20G05, 20G07, 20E28}
\dedicatory{Dedicated to the memory of Ernest Vinberg}

\maketitle

{\it Abstract} We study the spectra of non-regular semisimple elements in \ir \reps of simple
algebraic groups. More precisely, we prove that if $G$ is a simply connected simple linear
algebraic group and
$\phi:G\to {\rm GL}(V)$ is a non-trivial \irr for which there exists a non-regular non-central semisimple element
$s\in G$ such that $\phi(s)$ has
almost simple
spectrum, then, with few  exceptions,
$G$ is of classical type and $\dim V$ is minimal possible. Here the spectrum of a
diagonalizable matrix is called \emph{simple} if all \eis are of \mult 1, and \emph{almost simple}
if at most one \ei is
of \mult greater than 1. This yields a kind of characterization of the natural \rep (up to their
Frobenius twists) of classical algebraic groups in terms of the behavior of semisimple elements.

\medskip
Keywords: simple algebraic groups, representations, eigenvalue multiplicities, non-regular elements

\section{Introduction}\label{sec:intro}

A rather general problem which has  received attention in the literature can be stated as that of
classifying \ir group \reps whose image contains a matrix with a certain specified property. 
In this paper we concentrate on a property of the eigenvalue multiplicities of a semisimple
element of
simple linear algebraic groups in their \ir representations. (Henceforth we will use
``algebraic group''
to mean ``linear algebraic group''.) Although problems on eigenvalues in group representations
are important
for many applications, little can be said in full generality. In fact, the behavior of individual
elements in
the image of a representation is quite unpredictable. For a discussion of this and related
questions, we refer
the reader to \cite{Z}.

Here, we consider
matrices with almost simple spectrum, that is, matrices having at most one eigenvalue of
multiplicity
greater
than 1. More precisely, we will address the following:

\begin{prob}\label{p2} Let G be a simple  algebraic group defined over an algebraically closed
  field $F$.
Determine  the \ir representations $\phi$ of G such that $\phi(G)$ contains
a non-scalar diagonalizable matrix with almost simple spectrum. \end{prob}

Note that the notion of matrices with almost simple spectrum is  a natural generalization of
the similar notion of pseudo-reflections, the latter being diagonalizable matrices with two eigenvalues,
one of which has
\mult 1.
The classification of \ir matrix groups generated by  pseudo-reflections was an important project
enjoying numerous applications. (See \cite{Wag81,Wag78,SZ77,SZ80}.) We note as well that the consideration of Problem~\ref{p2} is an
extension of the analogous
question for finite quasi-simple groups of
Lie type and their representations in defining characteristic (see \cite{SZ98,SZ00}),
as well as the classification
(in \cite{GS,SZ1})
 of \ir \reps of simple
 algebraic groups for which a maximal torus acts with 1-dimensional weight spaces \cite{GS,SZ1}.
 A similar problem for irreducible representations of finite simple groups occurring
as subgroups of ${\rm GL}_n(\mathbb C)$ has been studied in \cite{KT}.

While Problem~\ref{p2} is a question about semisimple elements, there is a natural generalization
of the
notions of simple and almost simple
spectra to matrices that are not diagonalizable. Let $V$ be a finite-dimensional vector space
over a field $F$ and
$M\in {\rm GL}(V)$. Then $M$ is called \emph{cyclic} if, for some $v\in V$, the space $V$ is
spanned by
the vectors $v,Mv,M^2v,\ldots$, and \emph{almost cyclic} if, for some $\lam\in F$, $M$ is
conjugate to a
matrix ${\rm diag}(\lambda\cdot {\rm Id},M_1)$, where $M_1$ is a cyclic matrix. Almost cyclic matrices
in the
images of \irrs of finite simple groups
 are studied in \cite{DZ1,DZ3,DZP} (in certain special cases).
 Now let $g\in G$ and let $\phi$ be  an irreducible representation such that $\phi(g)$ is almost cyclic.
 If $g$ is
 not semisimple, then $g=su=us$ with $u\ne 1$ unipotent, and one sees that $\phi(u)$ has a single
 non-trivial
 Jordan block. Such representations have been determined in \cite{S-nonprime} and \cite{TZ16}. On
 the other
hand,  if $g$ is
semisimple, and $\phi(g)$ is almost cyclic, then $\phi(g)$ has almost simple spectrum; indeed
$\phi(g)$
has
at
most two eigenvalues, one of which has multiplicity 1.

Let us now return to our considerations of semisimple elements of $G$ whose spectrum in some
irreducible representation of $G$ is almost simple.
 As every semisimple element $s\in G$ lies in a
maximal
torus,  the condition for $\phi(s)$  to have simple spectrum  implies that all
weight
multiplicities of $\phi$ are equal to 1. The \ir \reps whose set of weights satifies this property
are
determined in 
\cite{GS} for tensor-indecomposable representations and completed in \cite{SZ1}.  By analogy,
one could
expect $\phi$ in Problem 1 to have all but one
weight multiplicity equal to 1. And indeed this is the case, as the following result, which will be
etablished in \S3, shows.

\begin{theo}\label{ag8} Let $G$ be a simple  algebraic group defined over an algebraically closed
  field $F$
  and $\phi$ an \irr of $G$. Then the following statements are equivalent:

$(1)$ The matrix $\phi(s)$ has almost simple spectrum for some non-central semisimple element $s\in G$.

$(2)$ All non-zero weights of $\phi$ are of \mult $1$.\end{theo}

Theorem~\ref {ag8} will be relevant to our consideration of Problem 1, especially
as the \ir \reps of simple algebraic group
satisfying (2)
have been determined in \cite{TZ2}. The above theorem is  best possible in the sense that in order
to obtain a more precise result one has to specify the nature of the semisimple element $s$
in question. We recall that an element $g\in G$ is said to be \emph{regular} if $\dim(C_G(g))$ is equal to the rank of $G$; for $g$
semisimple this is equivalent to $C_G(g)^\circ$ being abelian \cite[Chapter III, Corollary 1.7]{Spr}. Our investigations show that, with very few exceptions,
a non-central semisimple element $s$ having an almost simple spectrum in an \irr $\phi$ must be regular.

\begin{theo}\label{c99} Let $G$ be a simply connected simple  algebraic group defined over an algebraically
  closed field $F$ of characteristic $p\geq 0$ and let $s\in G$ be a non-regular non-central semisimple
  element.
  Let $V$ be a non-trivial irreducible $G$-module. If the spectrum of $s$ on $V$ is  almost
  simple,
  then 
one of the \f holds:
\begin{enumerate}
\item[\rm {(1)}] $G$ is   of Lie type $A_n,B_n~(p\neq 2)$, $C_n$ or $D_n$ and $\dim V=n+1,2n+1,2n,2n$,
  respectively;

\item[\rm{(2)}] $G=A_3$ and $\dim V=6;$

\item[\rm{(3)}] $G=C_2$, $p\ne 2$ and $\dim V=5;$

\item[\rm{(4)}] $G= D_4$ and $\dim V=8.$
\end{enumerate}
\end{theo}

The \ir \reps of $G$ of the dimensions given in Theorem~\ref{c99} are well known; a description of elements
$s$ which have  almost simple spectrum  on $V$ 
is provided in Section 3.

\medskip
{\bf Notation}  We fix an algebraically closed field $F$ of characteristic $p\geq 0$.  

Throughout the paper  $G$ is a simple simply connected linear algebraic group defined over $F$.
All $G$-modules considered are rational finite-dimensional $FG$-modules. For a $G$-module $V$
(or a \rep $\rho$ of $G$),
 we write $V\in\Irr(G)$ (or $\rho\in\Irr(G)$) to mean that $V$ (or $\rho$)  is rational and irreducible. If
 $H$ is a
 subgroup of $G$ then we write $V|_H$  for the restriction of a $G$-module $V$ to $H$.

We fix a maximal torus $T$ in $G$, which in turn defines the roots of $G$ as well as the weights of
$G$-modules and representations.
The $T$-weights of a $G$-module $V$ are the \ir constituents of the restriction of $V$ to $T$.
As $T$ is fixed, we will omit the reference to $T$ and write
``weights''
in place of ``$T$-weights''. The set of weights of $V$ is denoted by $\Omega(V)$. For $\mu\in\Omega(V)$,
the dimension of the $\mu$-weight space $\{v\in V: tv=\mu(t)v$ for all $t\in T\}$ is called
the
{\it \mult of $\mu$
  in} $V$.  The Weyl group of $G$ is denoted by $W(G)$; as $W(G)=N_G(T)/T$, the
conjugation
action of $N_G(T)$
on $T$ yields an action of $W(G)$ on $T$ and consequently on the set of $T$-weights. The
$W(G)$-orbit of
$\mu\in\Om$ is denoted by $W(G)\mu$. 
The set $\Omega={\rm Hom} (T,F)$ (the rational homomorphisms of $T$ to the multiplicative group of
$F$)
is called the {\it weight lattice}, which is a free $\ZZ$-module of finite rank
called {\it the rank of} $G$.

With an algebraic group $H$ is associated the Lie algebra of $H$ denoted here by ${\rm Lie}(H)$.
For the
simple group $G$, we denote the set of roots (that is, the non-zero weights of the $G$-module ${\rm Lie}(G)$)
by $\Phi$ or $\Phi(G)$. For notions of closed
subsystems of
$\Phi$ and subsystem subgroups see \cite[\S 13.1]{MT}.
The $\ZZ$-span of $\Phi$ is called the {\it root lattice} and is denoted here by $R$ or $R(G)$.
In $\Phi(G)$,
we fix
a base $\Pi=\{\al_1\ld \al_n\}$ and  order the simple roots according to the Dynkin diagrams as
in \cite{Bo}.
The weights in $R$ are called {\it radical}. For each root $\al\in \Phi(G)$, 
 we choose a non-zero element $X_\al$ in the $\al$-weight space of $T$ on
 $\Lie$. Thus, $FX_\alpha$ is the Lie algebra of a $T$-invariant one-dimensional unipotent
 subgroup $U_\al$ of
 $G $; see \cite[Theorem 8.16]{MT} for details.

 One defines a non-degenerate, $W(G)$-invariant, symmetric bilinear form on $\Omega\otimes_{\ZZ}{\mathbb R}$, which we
 express as
 $(\mu,\nu)$.
 The elements $\om_i$  satisfying $2(\om_j,\al_i)=(\al_i,\al_i)\delta_{ij}$ for $1\leq i,j\leq n$
 belong to
 $\Omega$
 and are called {\it fundamental dominant weights} \cite[Ch. VI,\S 1, no.10]{Bo}.  These form
 a $\ZZ$-basis
 of $\Omega$, so every $\nu\in \Omega$ can be expressed in the form $\sum a_i\om_i$, for $a_i\in\mathbb Z$; the set of
 $\nu$ with
 $a_1\ld a_n\geq 0$ is denoted by $\Omega^+$, the set of dominant weights. We set
 $\Omega^+(V)=\Omega^+\cap \Omega(V)$, so $\Omega^+(V)$ is the set of dominant weights of $V$.
 In what
 follows, we will regularly use so-called ``Bourbaki weights'', when $R(G)$ is of type
 $A_{r-1}, B_r, C_r$ or $D_r$, which are elements of a
 $\ZZ$-lattice
 containing $\Omega$ with basis   $\ep_1,\ep_2,\ldots,\ep_r$;  the explicit expressions of the fundamental
 weights and
 the simple roots of $G$ in terms of $\ep_i$'s are given in \cite[Planches I -- IV]{Bo}. 

 There is a standard partial ordering of elements  of $\Omega$:  for
 $\mu,\mu'\in\Omega$ we write
$\mu\prec\mu'$ and $\mu'\succ \mu$ \ii $\mu\ne\mu'$ and $\mu'-\mu\in R^+$. (We write $\mu\preceq \mu'$ and
$\mu'\succeq \mu$ to allow $\mu=\mu'$.) If $\mu$ and $\mu'$ are dominant weights such that $\mu'\preceq\mu$,
we say $\mu$\emph{ is subdominant to} $\mu$.  For the notion of a minuscule weight see
\cite[Ch. VIII, \S 7.3]{Bo8}, where they are tabulated.
Every \ir $G$-module has a unique weight $\om$ such that $\mu\prec\om$ for every $\mu\in\Omega(V)$ with
$\mu\neq \om$.
This is called the {\it highest weight of} $V$. There is a bijection between $\Omega^+$ and
$\Irr (G)$,
so for $\om\in\Omega^+$ we denote by $V_\om$ the \ir $G$-module with highest weight $\om$.
Suppose that $p>0$; a dominant weight $\sum a_i\om_i$ is called $p$-{\it restricted} if
$0\leq a_i<p$ for
all $i=1\ld n$. For uniformity,
we often do not separate the cases with $p=0$ and $p>0$; by convention, when $p=0$, a $p$-restricted weight is
simply
a dominant weight. An \ir $G$-module is called $p$-{\it restricted} if its highest
weight is
$p$-restricted. For classical groups $G$, that is, those with root system one of $A_n$, $B_n$, $C_n$ or $D_n$,
the module with highest weight $\om_1$ is called the
{\it natural module} and the associated representation the \emph{natural representation}. (There is an exceptional case, when $G=B_n$ and $p=2$, where the natural module is the Weyl module of highest weight $\omega_1$.)

The maximal height root of $\Phi(G)$ is denoted by $\om_a$; this is the highest weight of
 $\Lie$ and
 affords a
non-trivial composition factor of 
the adjoint module $\Lie$.
The short root module for $G$ of type $B_n,C_n,F_4,G_2$ is the \ir $G$-module all of whose non-zero
weights are short roots. This is unique, and the highest weight of the short root module is
maximal among
short roots (with respect to $\prec$).
An \ir $G$-module is called \emph{tensor-decomposable} if it is a tensor product of two or more
non-trivial
\ir modules,
similarly for representations.

If $h: G\ra G$ is a surjective algebraic group homomorphism  and $\phi$ is a representation of
$G$ then the
$h$-twist $\phi^h$ of $\phi$ is defined as the mapping $g\ra \phi(h(g))$ for $g\in G$. Of
fundamental importance
is the Frobenius mapping $Fr: G\ra G$ arising from the mapping $x\ra x^p$ $(x\in F)$ when $p>0$.
If $V$
is a  $G$-module and $k$ a nonnegative integer, then the modules $V^{Fr^k}$ are called {\it Frobenius twists of} $V$; if $V$ is
\ir with
 highest weight $\om$ then the highest weight of
 $V^{Fr^k}$ (for $k\geq 0$) is $p^k\om$.

 If $p=2$, then for
every $n$ there is a surjective algebraic group \ho $B_n\ra C_n$ with trivial kernel (so this is
an abstract
group isomorphism); for our purposes, the choice between these two groups is irrelevant, so we
choose to work
with $C_n$ when $p=2$.

For the natural $2n$-dimensional module $M$ of the group $C_n$, $n\geq 2$, 
a basis $\{e_i,f_i\ |\ 1\leq i\leq n\}$ is called \emph{ symplectic} $M$ if $\{e_i,f_i\}$ is a hyperbolic pair for all $i$ and $M$ is the orthogonal direct sum of the spaces $\langle e_i,f_i\rangle$, $1\leq i\leq n$.

Finally, we will assume $n\geq 1$ for $A_n$,  $n>1$ for $C_n$, $n>2$ for $G=B_n$, and $n>3$ for $D_n$.
For brevity
we write
$G=A_n$ to say that $G$ is a simple simply connected algebraic group of type $A_n$, and similarly
for the
other types.


\section{Preliminaries}

\begin{lemma}\label{ma2}  Let $M=M_1\otimes M_2$ be a Kronecker product of diagonal non-scalar
  matrices
  $M_1, M_2$ of sizes $m\leq n$, respectively. Suppose that M has almost simple spectrum.
  Then\begin{enumerate}[]

\item{\rm (1)}  $M_1$ and $M_2$ have simple spectrum, and

\item{\rm (2)} if
  $M_i$ is similar to $M_i\up$ for $i=1,2$, then the \ei multiplicities of $M$ do not exceed  $2$.
\end{enumerate}\el

\begin{proof} 
  (1) Suppose that $M_1$ has an \ei $e$, say, of \mult  $r>1$. Let $b_1,b_2$ be distinct \eis of
  $M_2$. Then
  $eb_1,$ $e b_2$ are distinct \eis of $M$, each of \mult greater than 1. This implies
the claim.

(2)  Suppose the contrary, and let $e$ be an \ei of $M$ of \mult at least 3. By (1), $M_1$ and
$M_2$ have
simple spectra  so $e=a_ib_i$ for $i=1,2,3$ and some (distinct) \eis $a_i$ of $M_1$ and $b_i$ of
$M_2$.
Then $e\up=a_i\up b_i\up$ is an \ei of $M$, of the same \mult as that of $e$. As $M$ has almost
simple
spectrum and is similar to $M^{-1}$ by hypothesis, we have $e=e\up$,
so  $a_1b_2=a_2\up b_1\up$. If $(a_2^{-1},b_2)\ne (a_1,b_1^{-1})$, then $a_1b_2$ is an eigenvalue
of $M$ of multiplicity 2 and so is equal to $e$. But this then implies $a_1b_2 = a_1b_1$,
contradicting that the $b_i$ are distinct. Hence $a_2 = a_1^{-1}$ and $b_2 = b_1^{-1}$.
Similarly, $a_1b_3 = a_3^{-1}b_1^{-1}$ implies that $a_3 = a_1^{-1}$ and $b_3 = b_1^{-1}$.
But now $a_2 = a_3$ contradicting that the $a_i$ are distinct. \end{proof}

\begin{defi}\label{se1} Let $V$ be a
  $G$-module and $\mu,\nu\in\Omega(V)$, $\mu\ne \nu$. 
We say that $s\in T$
\emph{separates the weights $\mu$ and $\nu$} if $\mu(s)\neq\nu(s)$. If this holds for every pair of
distinct weights
  $\mu,\nu$ of $V$, we say that $s$ \emph{separates the weights of} $V$. \end{defi}

If $s$  separates the weights of $V$ then the \ei multiplicities of $s$ acting on $V$ are simply the weight
multiplicities of $V$.

\begin{lemma}\label{nm1} Let $V$ be a non-trivial $G$-module. Let $S\subset T$ be the set of all $t\in T$ that
  separate the weights of $V$. Then \begin{enumerate}[] 
    \item{\rm{(1)}} $S$ is a nonempty Zarisky open subset of $T$.

    \item{\rm{(2)}} Suppose that at most one weight of $V$ has multiplicity greater than $1$. Then, for all
      $s\in S$, the
      spectrum of $s$ is  almost simple.
      \end{enumerate}
  \end{lemma}

  \begin{proof} $(1)$  Let $\mu,\nu$ be weights of $V$, $\mu\ne \nu$. Then $T_{\mu,\nu} :=\{x\in T\ |\ \mu(x)=\nu(x)\}$ is
    a Zarisky
    closed subset $T_{\mu,\nu}$ of $T$. The set of elements of $T$ that do not
    separate some pair of weights of $V$, being the finite union of all $T_{\mu,\nu}$,
    is a proper closed subset of $T$. Moreover, $S=T\setminus (\cup T_{\mu,\nu})$, and so (1) follows.

    $(2)$ Let $s\in S$, so that $\mu(s)\neq \nu(s)$ whenever $\mu\neq \nu$ are weights of $V$.
    Then the \eis of
    $s$ on $V$ are exactly $\mu(s)$, where $\mu$ runs over the weights of $V$,
    and the \mult of $\mu(s)$ equals that of $\mu$, giving (2).  \end{proof}
  
 We will require the following characterization of regular semisimple elements.

 \begin{propo}\label{re3} {\rm \cite[Ch. III, \S 1, Corollary 1.7]{Spr}} Let $G$, $T$ be as usual,
   and let $s\in T$. 
Then the \f conditions are equivalent:

$(1)$ $s$ is regular;

$(2)$ $C_G(s)$ consists of semisimple elements;

$(3)$ for all $\alpha\in\Phi(G)$, $\al(s)\ne1$;

$(4)$ $C_G(x)^\circ$ is a torus.
\end{propo}

 \begin{lemma}\label{td2} Let
 $V,V_1, V_2$ be non-trivial G-modules.
Let $s\in T\setminus Z(G)$ have almost simple spectrum on V.

 $(1)$ \St $V=V_1\otimes V_2$. Then 
 all weights of
$V_1$ and $V_2$ are of multiplicity $1$, and $s$ is regular.

$(2)$ Suppose that $\Om(V_1)+\Om(V_2)=\Om (V)$. Then $s$ separates the weights of $V_1$ and $V_2$.

\el

\begin{proof} The first claim of (1) follows from Lemma~\ref{ma2}. For the second assertion,
  suppose that $s$ is
  not regular.  Then by Proposition~\ref{re3},
  $C_G(s)$ contains a unipotent element $u\ne 1$. As $u$ stabilizes every eigenspace of $s$ on $V_1$, at
  least one of them is
  of dimension greater than $1$, contradicting Lemma~\ref{ma2}(1).

  (2) Suppose the contrary, that the weights of $V_1$, say, are not separated by $s$, so
  there exist distinct weights
$\mu_1,\mu_2\in\Om(V_1)$ such that $\mu_1(s)=\mu_2(s)$.
Then for every $\lambda,\mu\in\Omega(V_2)$, $\mu_i+\lambda,\mu_i+\mu\in\Omega(V)$ for $i=1,2$ and
$(\mu_1+\lam)(s)=(\mu_2+\lam)(s)$ and $(\mu_1+\mu)(s) = (\mu_2+\mu)(s)$. As $s\notin Z(G)$, the spectrum of
 $s$  on $V$ is not almost simple, a contradiction.\end{proof}

With regards to applying Lemma~\ref{td2}(2), we note that  $\Om(V) = \Om(V_1)+\Om(V_2)$ if $V=V_1\otimes V_2$.
For certain choices of $V,V_1, V_2$, and under certain conditions on $p$, we may deduce that
$\Om(V) = \Om(V_1)+\Om(V_2)$, for $V$ different from $V_1\otimes V_2$.  See Lemma~\ref{wtlattice}(2) below.






We recall here some basic facts about the set of weights of irreducible representations of a
simple algebraic group defined over a field of characteristic $0$ (which are derived from analogous
statements about the weights of irreducible representations of simple Lie algebras defined over $\mathbb C$).
Fixing a maximal torus $T_H$ of a simple algebraic group $H$ defined over ${\mathbb C}$, and
adopting the notation fixed earlier, so in particular, writing $W(H)$ for the Weyl group of $H$
relative to $T_H$, let $\lambda$ be a dominant $T_H$-weight. Then the set of weights of the irreducible
${\mathbb C}H$-module with highest weight $\lambda$ is precisely the set
 $$\{w(\mu)\ |\ \mu\in \Omega^+ , \mu\preceq\lambda,w\in W(H)\},$$
that is, the $W(H)$-conjugates of all weights which are subdominant to the highest weight $\lambda$.
From this one directly deduces the following facts:\begin{enumerate}
\item Let $\lambda,\mu\in \Omega^+ $ and $\mu\prec \lambda$. Let $V_\lambda$, respectively $V_\mu$
  be the associated
irreducible ${\mathbb C}H$-modules;  then $\Om(V_\mu)\subset \Om(V_\lambda)$.
\item \cite[Ch. VIII, \S 7, Proposition 10]{Bo8} Let $\lambda,\mu\in \Omega^+ $, with associated irreducible
  ${\mathbb C}H$-modules $V_\lambda$, $V_\mu$; then $\Om(V_{\lambda+\mu})=\Om(V_\lambda\otimes V_\mu)$.
\item \cite[Ch. VIII, \S 7, Propositions 4 and 6]{Bo8} Let $\lambda\in\Omega^+$, $\lambda\ne 0$.
  If $\lam$ is a radical weight, then some root is a weight of $V_\lam$; otherwise $\Omega(V_\lam)$ contains
  some minuscule weight.
\end{enumerate}

We now return to the situation where the field $F$ is of arbitrary characteristic.
We will use a fundamental result of Premet, which relies on the following definition and notation.

\begin{defn} We set 
$e(G)= 1$ for $G$ of types $A_n,D_n,E_n$, $e(G) = 2$ for $G$ of type $B_n,C_n,F_4$, and $e(G)=3$ for $G$
of type $G_2$. 

\end{defn}

\begin{theo}\label{premet} {\rm \cite[Theorem 1]{Pr}} Assume $p=0$ or $p>e(G)$.
Let $\lambda$ be a $p$-restricted dominant weight. Then
$\Omega(V_\lambda)=\{w(\mu)\ |\ \mu\in \Omega^+, \mu\preceq\lambda,w\in W(G)\}.$\end{theo}  An application of
Theorem~\ref{premet} and the preceding remarks now gives:

\begin{lemma}\label{wtlattice}
  Assume $p=0$ or $p>e(G)$. Let $\lambda,\mu\in \Omega^+$, where $\lambda$ is $p$-restricted, and let $V_\lambda$,
  respectively,
$V_\mu$ be the associated irreducible $G$-modules. Then the following hold.\begin{enumerate}[]
\item{\rm{(1)}} If $\mu\prec\lambda$ then
    $\Omega(V_\mu)\subseteq \Omega(V_\lambda)$.
  \item{\rm{(2)}} If $\lambda+\mu$ is $p$-restricted then   $\Omega(V_{\lambda+\mu})=
    \Omega(V_\lambda\otimes V_\mu)=\Omega(V_\lambda)+\Om( V_\mu)$.
  \item{\rm{(3)}} If $\lam$ is a radical weight, then some root is a weight of $V_\lam$; otherwise
    $\Omega(V_\lam)$ contains some minuscule weight.
\end{enumerate}
\end{lemma}

For the following result we introduce an additional notation. Let $\Psi\subset\Phi$ be a
closed subsystem. Then we set $G(\Psi)$ to be the subgroup generated by the
$T$-root subgroups corresponding to roots in $\Psi$.

\begin{theo}\label{sz5} {\rm \cite[Theorem 1]{SZ05}}
  Let $ G$ be a simple algebraic group with root system $\Phi$. If $\Phi$ is of type $B_n$, assume
  ${\rm char}(F)\ne 2$. Let $R_1,R_2\subset \Phi$ be closed subsystems
such that the subgroups $G_1:=G(R_1)$ and $G_2:=G(R_2)$ are simple and $[G_1,G_2]=1$.
Let $\phi$ be an \irr of $ G$. Then one of the \f holds:\begin{enumerate}[]

\item{\rm {(1)}} $\phi|_{G_1G_2}$ contains a composition factor which is non-trivial for both $G_1$
 and $G_2;$

\item{\rm{(2)}} $G$ is classical and $\phi$ is
a Frobenius twist of either the natural \rep or the dual of the natural representation of $G;$

\item{\rm {(3)}} $G=C_n$ with $p=2$, $G=B_n$ with $n>2 $, or $G=D_n$ with  $n>4$, and $\phi$ is a Frobenius
  twist of the irreducible representation of highest weight $\omega_n$, or one of $\omega_n $ and $\omega_{n-1}$ if $G=D_n$.\end{enumerate}\end{theo}

The \f lemma will allow us in some cases  to reduce our analysis of elements with almost simple spectrum
to \reps all of whose weights
 occur with multiplicity one.

 \begin{lemma}\label{no1} Let G be a simple algebraic group of rank greater than $1$ and
   $s\in T\setminus Z(G)$. 
   Assume that $p=0$ or $p>e(G)$. Let $\mu\neq 0$ be a $p$-restricted  dominant weight.
 \begin{enumerate}[]
 \item{\rm{(1)}} Let $\mu_m$ be the minimal non-zero weight  subdominant to $\mu$. Assume that the spectrum of $s$ on $V_{\mu_m}$ is not almost simple. Then the following hold:
   \begin{enumerate}[]
\item{\rm{(i)}} if $\mu$ is not radical, then  the spectrum of $s$ on $V_\mu$ is not almost simple;

\item{\rm{(ii)}} if $\mu$ is  radical and the \mult of the weight $0$  in
 $V_{\mu_m}$ is at most  $1$, then the spectrum of  $s$ on $V_\mu$  is not almost simple;

\item{\rm{(iii)}} if $\mu$ is  radical and the \mult of the weight $0$  on both
$V_{\mu}$ and  $V_{\mu_m}$
is greater than $1$, then the spectrum of $s$ on $V_\mu$ is not almost simple;

\item{\rm{(iv)}} if $0\prec \mu_m\preceq\mu$ and $s$ is non-regular,
 then the spectrum of  $s$ on $V_{\mu}$ is not almost simple.
\end{enumerate}

\item{\rm{(2)}} Suppose that $\omega_a\prec\mu$,
   the \mult of the weight  $0$ in $V_\mu$ is greater than $1$, 
   and  the spectrum of  $s$ on $V_{\om_a}$  is not almost simple. Then the spectrum of  $s$ on $V_{\mu}$ is not
   almost simple.

 \item{\rm {(3)}} Suppose that $\om_a\prec\mu$, $s$ is non-regular,
   and  the spectrum of  $s$  on $V_{\om_a}$ is not almost simple. Then the spectrum of  $s$ on $V_{\mu}$  is
   not
   almost simple.
\end{enumerate} \el

 \begin{proof} By assumption, Theorem~\ref{premet} holds, and we may apply Lemma~\ref{wtlattice}.
   If $\mu_m$ is non-radical, then all weight multiplicities of $V_{\mu_m}$   are well known to be equal to
   $1$; (i) follows.
   Together with the hypothesis in (ii) about the multiplicity of the zero weight, we observe that if the
   spectrum of
$s$ on $V_{\mu_m}$ is not almost simple then there are $4$ distinct weights $\lam_1,\lam_2,\mu_1,\mu_2$ of
$V_{\mu_m}$ such that $\lam_1(s)=\lam_2(s)\neq \mu_1(s)=\mu_2(s)$. Then Lemma~\ref{wtlattice}(1) implies
that  these weights are weights of $V_{\mu}$, and the result follows.

In case (iii), $\mu_m$ is the maximal height short root and  the multiplicity of any non-zero weight
in $V_{\mu_m}$   is  equal to $1$.  Saying that  the spectrum of  $s$ on $V_{\mu_m}$  is not almost simple means that there
exist weights $\lam_1,\lam_2 $ of $V_{\mu_m}$ such that $\lam_1(s)=\lam_2(s)\neq 1$. As these weights are
weights of $V_{\mu}$ (again by Lemma~\ref{wtlattice}(1)) and, by hypothesis, the weight $0$ occurs in $V_{\mu}$
with
\mult greater than $1$, the result follows.

(iv) As $s$ is non-regular, there exists $\al\in\Phi(G)$ such that $\pm\al(s)=1$ (Proposition~\ref{re3}). Since  the spectrum of $s$  on $V_{\mu_m}$ is not almost simple, there are distinct short roots $\beta,\gamma$ such that $\beta(s)=\gamma(s)\neq 1$.  Then Lemma~\ref{wtlattice} implies that $\pm\al,\beta,\gamma$ are weights of $V_\mu$, and the result follows.

For (2), first note that the \mult of the weight $0$ in $V_{\omega_a}$ is greater than $1$ unless
$(G, p)=(A_2,3)$ (here we again rely on the prime restrictions in the hypotheses). This case
is considered in (ii). In all other cases, saying that  the spectrum of  $s$ on $V_{\om_a}$ is not almost
simple  means that  there are two roots $\al,\beta$ such that $\al(s)=\beta(s)\neq 1$.  As the weights of
$V_{\omega_a}$ occur as weights of $V_{\mu}$ and the weight $0$ occurs in $V_{\mu}$ with \mult greater than $1$,
the result follows.

Finally, the case  (3) follows as (iv) above, where one has to replace $V_{\mu_m}$ by $V_{\om_a}$ and
``short roots''
by ``roots''. \end{proof}

We complete this section with a straightforward observation about the natural modules for classical groups.

\begin{lemma}\label{co1} Let $G$ be a classical type group and assume $p\ne 2$ when $G$ is of type $B_n$. Let
$V=V_{\om_1}$ and  $s\in G$ be a non-central semisimple element.

$(1)$
For $G=A_n$ or $C_n$, if $s$ is regular, then $s$ has simple spectrum on $V$

$(2)$  Let $G=B_n$. Then $s$ is regular \ii the \mult of the \ei $-1$ on $V$ is at most $2$ and the other \ei multiplicities are equal to  $1$.

$(3)$ Let  $G=D_n$. Then $s$ is regular \ii the multiplicities  of the \eis $1$ and $-1$ on $V$
are at most $2$ and the other \ei multiplicities are equal to  $1$. In addition, if the spectrum of $s$
on $V$ is not almost simple then that of $s$ on $V_{\om_2}$ is not almost simple.

$(4)$ If $s$ is  regular then the spectrum of $s$   on $V$ is almost simple unless $G=D_n$, $p\neq 2$ and
$1,-1$ are \eis of $s$ on V, each of \mult $ 2$. \el

\begin{proof} (1) This is straightforward and well known.

For the remainder of the proof, we take $T$ to be the maximal torus consisting of the diagonal matrices in the image of the natural representation of $G$. We now turn to  (2) and the first statement of (3). Observe that $\Om(V)$ consists of
the weights $\pm\ep_i$, $1\leq i\leq n$, together with the weight $0$ in case  $G=B_n$.
In addition, $s$ is regular \ii $\al(s)\neq1$ for every root $\al$.
Set $a_i=\ep_i(s)$ and recall that $\Phi(D_n) =  \{\pm\ep_i\pm\ep_j\ |\ 1\leq i<j\leq n\}$ and
$\Phi(B_n) = \{\pm\ep_i\pm\ep_j, \pm\ep_r\ |\ 1\leq i<j\leq n,1\leq r\leq n\}$.
So $s$ is regular \ii $a_i\neq a_j$ and $a_i\neq a_j\up$ for every $i\neq j$, and if in addition, for $G=B_n$, $a_i\neq 1$
for all $1\leq i\leq n$. So if $G = B_n$, we see that $s$ is regular if and only if either all of the eigenvalues
$a_1^{\pm 1}, a_2^{\pm1},\dots,a_n^{\pm1}$ are distinct and distinct from $1$, or there exists a unique $i$ with $a_i=a_i^{-1}$. If $a_i=a_i^{-1}=- 1$, then  $s$ is regular if and only if all eigenvalues of $s$ on $V$ different from $-1$ occur with multiplicity $1$,
and
$-1$ occurs with multiplicity at most $2$. Now if $G=D_n$, then $s$ is regular if and only if $a_1^{\pm},\dots,a_n^{\pm}$ are distinct or there exists $1\leq i\leq n$ such that $a_i=a_i^{-1}$. In the latter case, $s$ is regular if and only if  all eigenvalues different from $a_i$ occur with multiplicity $1$ and
the eigenvalue $a_i$ can occur with multiplicity at most $2$, as claimed.

For the final statement of (3), let $G=D_n$ and suppose that the spectrum of $s$ on $V$ is not almost simple.
 Then, without loss of generality, we may assume
  $a_i=a_j$ for some $1\leq i\ne j\leq n$.
 Then $(\ep_i-\ep_k)(s)=(\ep_j-\ep_k)(s)$ and $(-\ep_i-\ep_k)(s)=(-\ep_j-\ep_k)(s)$ for every $k\neq i,j$.
 Recall that the non-zero weights of $V_{\omega_2}$ are the roots in $\Phi(G)$, and the zero weight occurs with
 multiplicity at least $2$.
 Assume for a contradiction that the spectrum of $s$  on $V_{\om_2}$ is almost simple.
 Then $(\ep_i-\ep_k)(s)=(\ep_j-\ep_k)(s)=(-\ep_i-\ep_k)(s)=
 (-\ep_j-\ep_k)(s)=1$, whence $-\ep_i(s)=\ep_i(s) = \ep_k(s)$ for all $1\leq k\leq n$.
 As $s\notin Z(G)$, we get a contradiction.

(4) This follows from $(1),(2)$ and $(3)$. 
\end{proof}



\section{Reduction theorem, and  proof of Theorem~\ref{ag8}}\label{sec:reduction}

Let $S\subset {\rm GL}(V)$ be an abelian subgroup and let $\Irr(S)$ denote the set of
irreducible $F$-linear representations of $S$ and write $1_S$ for the trivial representation. For $\eta\in \Irr S$, set
$V_S(\eta)=\{v\in V: sv=\eta(s)v$ for all $s\in S\}$. If $V_S(\eta)\ne \{0\}$, we say $\eta$ is an $S$-weight of $V$ and
we call
$V_S(\eta)$ the \emph{$\eta$-weight space for $S$}. As throughout $G$ is a  simple  algebraic group defined
over $F$ and $
T\subset G$ is a maximal torus of $G$. If $V$ is a rational $G$-module then $V$ is a direct sum of
$T$-weight spaces and
for any subgroup $S\subseteq T$, these weight spaces are $S$-invariant.
Thus for $\eta\in\Irr(S)$, $V_S(\eta)$ is a sum of $T$-weight spaces of $V$.
We establish here a result about such subgroups $S$ of $T$, and later will apply this to the case where
$S$ is the subgroup generated by an element $s\in T$.

Recall (see for instance \cite[\S 7]{MT}) that for
any rational representation $\rho:G\to{\rm GL}(V)$, we have a corresponding representation of $\Lie$, namely
$d\rho: \Lie\to {\rm {Lie}}(\GL(V))$.
For $g\in G$, let $t_g:G\to G$ denote the automorphism induced by conjugation by $g$. Then using the basic
definitions and
properties of the differential, we have that $t_{\rho(g)}\circ\rho = \rho\circ t_g$ and so
$${\rm Ad}(\rho(g))\circ d\rho = d\rho\circ{\rm Ad}(g).$$

\begin{theo}\label{ac4} {\rm (Reduction theorem)}
 Let $G$ be a simple  algebraic group, $T$ a maximal torus of $G$, and
$S\subseteq T$ a  subgroup such that $C_G(S)\neq G$.
Let $V$ be an \ir $G$-module with $p$-restricted highest weight. 
Let $V_S(\eta)$ be an $S$-weight space of $V$, for some  $\eta\in \Irr S$. Suppose
that $\dim V_S(\eta)=k>1$ and that all other $S$-weight spaces on $V$
are of dimension $1$. Then all non-zero $T$-weights of $V$ are of \mult $1$. \end{theo}

\begin{proof} Set   $E=V_S(\eta)$. For $\mu\in\Omega(V)$, write $M_\mu$ for the
  $T$-weight space of $V$ associated to $\mu$. Suppose that $\dim M_\mu\geq 2$, for some
  $\mu\in \Omega(V)$. Then $M_\mu\subset E$. As $\dim M_\mu=\dim
  M_{w(\mu)}$ for any $w\in W$, we necessarily have $M_{w(\mu)}\subset E$.
  Now let $\rho:G\to {\rm GL}(V)$ be the corresponding rational representation of $G$. For a root
  $\alpha\in \Phi$, $\alpha$
  induces a $1$-dimensional
representation $\lambda_\alpha$ of the group $S$.

Consider first the case where $\lambda_\alpha\ne 1_S$, for all $\alpha\in \Phi$.
Recall the notation $X_\alpha\in \Lie$, a root vector associated to the root $\alpha$, a
fixed element which spans the Lie algebra of the associated root group.
Then $d\rho(X_\alpha) E\subset V_S(\eta\lambda_\alpha)$.
Since $\eta\lambda_\alpha\ne \eta$, this latter $S$-weight space is of dimension
at most $1$. Hence $K_\alpha: = \ker((d\rho(X_\alpha))\mid_E)$ is of dimension at least
$k-1$. Setting $K_1 = \cap_{\alpha\in \pm\Pi}K_\alpha$, we see that $K_1\subset V$ is a
proper $\Lie$-submodule on which
$\Lie$ acts trivially. But by \cite{Cu}, $V$ is an irreducible $\Lie$-module,
and so  $K_1=\{0\}$.  Therefore, $k=\dim E\leq 2n$, where $n$ is the rank of $G$.
We can now show that $\mu=0$; for otherwise
the $W$-orbit of $\mu$ is of length at least $n+1$ (the exact
values are in the \cite[Table 1]{Z}).    Therefore, $\dim \sum
M_{w(\mu)}\geq 2(n+1)$, which is a contradiction.

Consider now the case where  there exists $\alpha\in\Phi$ such that $\lambda_\alpha=1_S$.
Set $M':=\sum_{w\in W} M_{w(\mu)}$, so that $M'\subseteq E$. Let
$R_0=\{\al\in\Phi: \lambda_\al=1_S\}$, $R_2=\Phi\setminus R_0$. Since $S$ is non-central,
$R_0\ne \Phi$ and $R_2\ne\emptyset$.
Let $R_1$ be the set of roots $\al$ such that $\dim (d\rho(X_\al) M')\leq 1.$  By the considerations of
the first case above, $R_2\subseteq R_1$. Moreover, we claim that $R_1$ is $W$-stable. Indeed for
$w\in W$, choose  $\dot{w}\in N_G(T)$ such that $w=\dot{w}T$. Then $$\rho(\dot{w})d\rho(X_\alpha) M' =
\rho(\dot{w})d\rho(X_\alpha)\rho(\dot{w})^{-1}\rho(\dot{w})M' =
{\rm Ad}(\rho(\dot{w}))(d\rho(X_\alpha)) M'.$$ By the remarks preceding the statement of the result, this
latter is equal
to $$ d\rho({\rm Ad}(\dot{w})X_\alpha) M' = d\rho(X_{w(\alpha)})M'$$ and since
$\dim (\rho(\dot{w})(d\rho(X_\al) M'))=\dim (d\rho(X_\al) M')$, we have the claim. Now, if all roots of
$\Phi$ are
of the same length then $R_1=\Phi$, and we conclude as in
the first case.

Hence we may assume that $\Phi$ has two root lengths and  that  the roots of $R_1$ are of a single length.
Note that $R_0=-R_0$ and $\beta,\gamma\in R_0$ implies $\beta+\gamma\in R_0$ provided
$\beta+\gamma $ is a root.  This implies (see for example \cite[B.14]{MT}) that $R_0$ is a root system,
that is, $R_0$ is a closed subsystem of $\Phi$.
Moreover, $R_0$ is of maximal rank (equal to the rank of $\Phi$) as otherwise, by \cite[B.18]{MT},
$R_0$ lies in some subsystem corresponding to a proper subset of $\Pi$, in which case $R_2$, and so
$R_1$ has roots of both lengths.  So  $R_0$ is a subsystem of maximal rank, and by the classification of such,
\cite[B.18]{MT}, one checks that in every case
$\Phi\setminus R_0=R_2$ again contains roots of both lengths and we conclude as above.\end{proof}

\begin{rem}\label{nonrest} If $\om=p^k\om'$, with $\om'$ $p$-restricted,
then the weights of $V_\om$ are $p^k\mu$ for $\mu$ a weight of $V_{\om'}$. Then $p^k\mu(s)=\mu(s^{p^k})$.
As the mapping $x\ra x^{p}$ for $x\in F$ is bijective on $F$, the spectrum of $s$ on $V_\om$ is almost simple \ii the spectrum of $s$ on $V_{\om'}$ is almost simple.

\end{rem}

We now take $S$ to be generated by a
single element $s\in T$ and consider the case of tensor-decomposable irreducible representations.

\begin{lemma}\label{gc3} Let $s\in T$ be a
non-central element. Let $\om$ be a dominant weight which is not $p$-restricted and
not of the form $p^k\mu$ for $\mu$ a $p$-restricted weight. Suppose that the spectrum of $s$
 on $V_\om$ is almost simple. Then all weights of $V_\om$ are of \mult $1$.\el

\begin{proof} 
By Steinberg's tensor product theorem,
$V_\om=V_{p^{k_1}\mu_1}\otimes V_{p^{k_2}\mu_2}\otimes \cdots \otimes V_{p^{k_t}\mu_t}$, where $t>1$ and
$\mu_1\ld \mu_k$ are non-zero $p$-restricted weights and $(k_1,\dots,k_t)$ are distinct non-negative integers.
Then Lemma~\ref{td2} implies that the spectrum of $s$  on each  tensor factor is simple so the weights of each tensor factor have
multiplicity $1$.  Furthermore, \cite[Proposition 2]{SZ1} implies that the weights of $V_\om$ are of \mult 1 unless
there exists $1\leq j<t$ such that $k_{j+1} = k_j+1$ and one of the following holds:
\begin{enumerate}[]
\item{\rm{(i)}} $G=C_n$, $p=2$, $\mu_j=\om_n$, $\mu_{j+1}=\om_1;$
\item{\rm{(ii)}} $G=G_2$, $p=2$, $\mu_j=\om_1$, $\mu_{j+1}=\om_1;$
\item{\rm{(iii)}} $G=G_2$,  $p=3$, $\mu_j=\om_2$, $\mu_{j+1}=\om_1$.
\end{enumerate}

Moreover, in each of the cases (i), (ii) and (iii), the module $V_{\mu_j}\otimes V_{p\mu_{j+1}}$ has a weight of multiplicity
greater than $1$. Hence if one of the three cases occurs, we deduce that $t=2$ and so we can also assume
that $j=1$ and $k_1=0$, that is, $V_\om=V_{\mu_1}\otimes V_{p\mu_2}$. We consider the above cases in detail.

Case (i): Take $T$ to be the set of diagonal matrices in the image of the natural representation of $G$.
Here  $\Om(V_{\om_n})= \{\pm \ep_1\pm\cdots\pm\ep_n\}$ and  $\Om(V_{2\omega_1})=\{\pm 2\ep_1\ld \pm2\ep_n\}$.
(As usual, we have adopted the notation of \cite[Planche III]{Bo}.) Let $\nu$ be a weight of $V_{\om_n}$ with positive
signs of both $\ep_i$ and $\ep_j$, for some $1\leq i,j\leq n$, $i\ne j$. As $\nu-2\ep_i$ and $\nu-2\ep_j$ are weights of $V_{\om_n}$, it follows that $\nu$ is
also a weight of $V_{\om}$ with \mult at least 2. This remains true for weights
where both $\ep_i$ and $\ep_j$ have coefficient $-1$ or have opposite coefficients.
  \itf the restriction of $V_{\om_n+2\om_1}$ to $T$ contains a direct sum of at least two  copies of
$V_{\om_n}|_T$. Therefore, every \ei of $s$ on $V_{\om_n}$ is also an
\ei of $s$ on $V_{\om_n+2\om_ 1}$, and occurs with \mult at least $2$. So this case is ruled out as the spectrum of $s$  on
$V_{\omega_n}$ is simple.

Case (ii): Here the weights of $V_{\omega_1}$ are the short roots of $\Phi$, and the following weights occur with
multiplicity $2$ in $V_\om$: $3\alpha_1+\alpha_2$, $3\alpha_1+2\alpha_2$.
Since the spectrum of  $s$ on $V_\om$ is  almost simple, these roots must all take equal value on $s$. In 
particular,
$\alpha_2(s)=1$. But now the eigenvalue $5\alpha_1(s)$ occurs with multiplicity $2$ as well as
$3\alpha_1(s)$, implying that $\alpha_1(s)=1$ as well, contradicting the fact that $s$ is non-central.

Case (iii): This case is similar. Here the weights of $V_{\om_2}$ are the long roots of $\Phi$ and the
 zero weight, and the weights of $V_{\om_1}$ are the short roots and the zero weight. We find that
each of the weights $3\alpha_1+\alpha_2$ and $\alpha_2$ occur with multiplicity $2$, and deduce that
$\alpha_1(s)=1$. But now the eigenvalue $\alpha_2(s)$ occurs with multiplicity greater than $1$, as well as the
eigenvalue $1$, and so $\alpha_2(s)=1$ as well, again contradicting $s$ non-central. \end{proof}


\begin{proof}[Proof of Theorem~$\ref{ag8}$]   Using Lemma~\ref{nm1}, we see that assertion (1) follows from assertion (2).
  We apply Theorem~\ref{ac4}, Remark~\ref{nonrest} and Lemma~\ref{gc3} to obtain the reverse implication. \end{proof}


\section{Commuting subgroups and a partial proof of Theorem~\ref{c99}}\label{sec:com}

An essential element of our proof of Theorem~\ref{c99} is an application of Theorem~\ref{sz5}, which allows
us to treat many of the groups and representations in a uniform way. (See Proposition~\ref{vg1} below.)
Let $s\in G$ be a non-regular semisimple element. In order to apply Theorem~\ref{sz5}, we need to
find a pair of subsystem subgroups $K,Y$ such that $[K,Y]=1$, $[K,s]=1$ and $[s,Y]\neq 1$. For technical
reasons,
 it will suffice to do this for groups other than $B_n,D_n$, and $G_2$.

 \begin{lemma}\label{ab1} Let $G=\SL_n(F)$, $n>3$, and let $s\in T\setminus Z(G)$ be a non-regular element.
 Then there are simple subsystem subgroups $K,Y$, normalized by $T$, such that $[K,Y]=1$, $[K,s]=1$
and $[s,Y]\neq1$, unless $n=4$ and, up to conjugacy in $G$,
$s=\diag(a,a,a\up,a\up)$ or   $s=\diag(a,a,-a\up,-a\up)$.\end{lemma}

\begin{proof} We take $T$ to be the torus of diagonal matrices in $G$. As $s$ is non-regular and non-central, we may assume that $s=\diag(b,b,a_3\ld a_n)$, where
  $a_3\neq b$.
Suppose first that $a_3\neq a_i$ for some $i>3$. Set $K=\diag (\SL_2(F), \Id_{n-2})$, $Y=\diag({\rm Id}_2,\SL_{n-2}(F))$.
Next, suppose $a_3=\cdots =a_n$. If $n>4$ then we can take
$Y=\diag(1,\SL_{2}(F),\Id_{n-3})$ and $K=\diag(\Id_{n-2}, \SL_2(F))$. If $n=4$, then $s=\diag(b,b,a,a)$ and $b^2a^2=1$,
whence
$b=\pm a\up$.\end{proof}

\begin{remk} If $G=\SL_4(F)$, and
  $s=\diag(\lam,\lam,\lam^{-1},\lam^{-1})$ or $s=\diag(\lam,\lam,-\lam^{-1},-\lam^{-1})$, for $\lam\in F$,  $\lam^4\neq 1$,
  then $s$ is non-regular, non-central, and it is impossible to find a pair of subsystem subgroups  $K$, $Y$ such that
  $[s,K]=1$ and $[s,Y]\neq1$.
 Moreover, the Jordan form of $s$ on the exterior square of the natural $4$-dimensional module
 is $\diag(\lam^2,\lam^{-2},1,1,1,1)$, which is non-central with almost simple spectrum.\end{remk}

\begin{lemma}\label{a12} Let $G=C_n$, $n>1$, and let $s\in T\setminus Z(G)$ be a non-regular element.
 Then there are simple subsystem subgroups $K$, $Y$ of $G$, normalized by $T$, such that $[K,Y]=1$, $[K,s]=1$
 and $[s,Y]\neq1$, unless $n=2$ and  with respect to  an ordered  symplectic basis  $(e_1,f_1,e_2,f_2)$
 of $V_{\omega_1}$,
the Jordan form of s on the natural G-module is  either $\diag(a,a\up,a,a\up)$, for $\pm 1\neq a\in F$,  or
$s=\pm\diag(1,1,-1,-1)$, for $p\ne 2$.\end{lemma}

\begin{proof} The group $G=C_n=\Sp_{2n}(F)$ contains a maximal rank subsystem subgroup $H$ isomorphic to
  $\Sp_2(F)\times \cdots \times \Sp_2(F)$, so every semisimple element is conjugate to an element of $H$.
  Therefore, we can  write the matrix of $s$ with respect to a suitable basis of the natural $G$-module $V_{\omega_1}$ as $\diag(a_1,a_1\up\ld a_n,a_n\up)$ for some $a_1\ld a_n\in F$.
  By Lemma~\ref{co1}, the diagonal entries of $s$ are not distinct.  Hence either $a_i=\pm 1$
  for some $i\in\{1\ld n\}$, or, replacing some $a_i$ by $a_i\up$, we can assume that $a_i=a_j$ for some
  $1\leq i<j\leq n$.

Suppose first that $a_i=\pm 1$ for some $i\in\{1\ld n\}$ and assume without loss of generality
that $i=1$. If there exists $j$ such that $a_j\neq \pm 1$,  we can assume $j=n$ and then take
$K=\diag(\Sp_2(F),\Id_{2n-2})$,
$Y=\diag(\Id_{2n-2}, \Sp_2(F))$. Otherwise, $s^2=1$ and $p\ne 2$. We can reorder $a_1\ld a_n$ so that
$a_1\neq a_2$, and if $n>2$ we take $Y=\diag(\Sp_4(F),\Id_{2n-4})$, $K=\diag(\Id_{2n-2}, \Sp_2(F))$.  If $n=2$, $s^2=1$ and $p\ne 2$, such a choice is not possible and we
have $s$ as in the final statement. 

Now suppose that $a_i\neq \pm 1$ for all $i\in   \{1\ld n\}$, so there exists $1\leq i < j\leq n$
such that $a_i=a_j$. In this case, there exists a $2$-dimensional totally isotropic subspace of the
underlying $2n$-dimensional symplectic space on which $s$ acts as scalar multiplication.
If $n>2$, then $s$ is contained in a Levi subgroup $L=L_1\times L_2$ of $G$, where
$L_1\cong {\rm GL}_2(F)$ and $L_2\cong \Sp_{2n-4}(F)$.  Moreover $[s,L_1]=1$, so we can take
$K=L_1$, $Y=L_2$. If $n=2$ then $s=\diag(a,a\up,a,a\up)$ as in the statement of the result.
\end{proof}

\begin{lemma}\label{ab2} Let $G\in\{E_6,E_7,E_8, F_4\}$. Let $s\in T\setminus Z(G)$ be a non-regular element.
  Then there  exist simple subsystem  subgroups $K$, $Y$, normalized by $T$, such that $K$ is of type $A_1$,
$[K,Y]=1$, $[K,s]=1$, $[s,Y]\neq 1$. \end{lemma}

\begin{proof} As $s$ is not regular, $C_G(s)$ contains root subgroups $U_{\pm\alpha}$ for some root
  $\al\in\Phi$.
  Clearly, we can
  assume $\al$ to be a simple root. Moreover, we can assume that
 $\al=\al_1$  if $G\neq F_4$,   otherwise, that $\al=\al_1$ or $\al_4$.

Denote by $R_\al$  the set of  roots orthogonal to $\al$, and observe that $R_\al$ is not empty.
Set $Y=\lan U_{\pm \beta}:\beta\in R_{\al}\ran$ and $K=\langle U_{\pm\alpha}\rangle$. Then $[Y,K]=1$ and
$[K,s]=1$.
If $[Y,s]\neq 1$, replacing $Y$ by a suitable simple subgroup of $Y$, we are done.

We now assume $[s, U_\beta]=1$ for  all $\beta\in R_\al$. In this situation, as $s$ is non-central,
$[s,U_\gamma]\neq 1$
for some simple root adjacent to $\al$ in the Dynkin diagram. Moreover,
the Dynkin diagram of the above groups contains a node $\beta$,  not adjacent to each of $\al,\gamma$.
In particular, $\beta\in R_\alpha$ and so  $[s,U_{\beta}]=1$, while  $[s,U_\gamma]\neq 1$. So now we can take
$K=\lan U_{\pm\beta}\ran$ and $Y=\lan U_{\pm\gamma}\ran$.

This completes the proof.\end{proof}

We now apply the previous three lemmas and Theorem~\ref{sz5} to establish Theorem~\ref{c99} for certain groups.

 \begin{propo}\label{vg1} Let $G$ be of type $A_n$ for $n>3$, $C_n$ for $n>2$, or of type $F_4, E_6, E_7$, or $E_8$.
 Let $V$ be a non-trivial \ir   $G$-module and
  $s\in T\setminus Z(G)$. Suppose that the spectrum of s  has on V is almost simple.
  Then one of the following holds:
\begin{enumerate}[]
\item{\rm{(1)}}  $s$ is  regular,
\item{\rm {(2)}} $G=C_n$ with $p=2$ and the highest weight of V is $2^m\om_n$, or
\item{\rm {(3)}} $G$ is classical and $V$ is a Frobenius twist of the natural or the dual
of the natural module for $G$.
\end{enumerate}
\end{propo}

\begin{proof} Suppose that $s$ is not regular. By Lemma~\ref{ab1} for $A_n$, Lemma~\ref{a12} for $C_n$,   and
Lemma~\ref{ab2} for the other groups in the statement, there are simple subsystem subgroups $K$, $Y$,
normalized by $T$, such that $[K,Y]=1$, $[K,s]=1$ and $[Y,s]\neq 1$.
 Then we apply Theorem~\ref{sz5} to $K$, $Y$ in place of $G(R_1),G(R_2)$ to conclude that either
(2) or (3) holds or there is a
$KY$-composition factor $M$ of $V$ afforded by an irreducible representation $\tau$
of $KY$, such that $\tau$ is non-trivial on both $K$ and $Y$. So we assume neither (2) nor (3) holds,
so  we are in the latter situation, and aim for a contradiction.

We first note that $TY=Y\cdot Z(TY)$, as $Y$ is simple. Therefore, as $s\in T$,
$s=s_1s_Y$ for some $s_1\in  Z(TY)\subset T$ and $s_Y\in (T\cap Y)$.
As $[s,K]=1$ and $[Y,K]=1$, we have $[s_1,K]=1$ and  $[s_1,YK]=1$. Also, as $[s,Y]\ne 1$, we have
$[s_Y,Y]\ne 1$.

Now $M$ is a direct sum of eigenspaces for $s_1$. \itf $\tau$ is realized in one of the $s_1$-eigenspaces
$M_1$, say, and hence the spectrum of $s$ on $M_1$ is almost simple   \ii that of $s_Y$ on $M_1$ is almost
simple. Therefore, it suffices to show that the spectrum of $\tau(s_Y)$ is not almost simple.

Now $\tau=\tau_K\otimes \tau_Y$, where $\tau_K$, $ \tau_Y$
are non-trivial \ir  \reps of $K$, $Y$, respectively. As $[s_Y,Y]\neq 1$,
there are at least two distinct $s_Y$-eigenspaces  on the representation
space corresponding to $\tau_Y$, each of  them  is of dimension at least $2$ as $\tau_K(K)$ acts
on each eigenspace and all $\tau_K(K)$ composition factors of $M$ are of dimension strictly greater
than $1$.
Hence, the spectrum of $s_Y$ on $M$  is not almost simple, giving the desired contradiction. \end{proof}

\begin{remk}\label{rm5} $(1)$ Let $G=C_2$, $p$ odd. If s is not as described in the exceptional cases of
  Lemma~\ref{a12} then Proposition~\ref{vg1} remains valid.

  $(2)$ Note that the \irr of $G=C_2$ with highest weight $\om_2$ induces an isomorphism between
  ${\rm PSp}_4(F)$  and ${\rm SO}_5(F)$,  and the  element  $s=\pm\diag(1,1,-1,-1)$ in Lemma~\ref{a12}
  acts as $\diag(1,-1,-1,-1,-1)$, hence has almost simple  spectrum.  Similarly, the
  element $s=\diag(a,a\up,a,a\up)$ acts as $\diag(a^2,1,1,1,a^{-2})$, which has almost simple spectrum provided
  $a^2\neq \pm1$.

(3)
In view of Lemma~\ref{td2} and Proposition~\ref{vg1}, to complete the proof of  Theorem~\ref{c99},
it remains to consider $p$-restricted representations (of highest weight $\lambda$) of the groups $B_n$ for
$n>2$, $D_n$  for $n>3$, $C_n$ for $p=2$ and $\lam=\om_n$, and the small rank groups $A_2$, $A_3$, $C_2$, and
$G_2$. We will  handle the small rank groups in Section~\ref{sec:smrk} and complete the proof in
Section~\ref{BCD} by dealing with the remaining groups. \end{remk}


\section{Weight levels}\label{sec:levels}

Set $\Lambda = \sum_{i=1}^n{\mathbb Z}\omega_i$, the weight lattice associated with $\Phi$, and $\Lambda^+$ the set of
dominant weights in $\Lambda$. In this section we
establish some results on $\Lambda$ in view of applying the results in Section 2.  Recall that a weight
is  radical if it is an integral linear combination of roots.
 The \ir $G$-module whose highest weight is the maximal height short root is called the short root module. If all weights are of the same length then any root is regarded as  short, and the short root module  is $V_{\om_a}$.   

\begin{defn} Let $$\Lambda_{1}=\{\mu\in\Lambda^+\   | \ \mbox{ if } \nu\preceq\mu\mbox{ for some }
 \nu\in\Lambda^+\mbox{ then }  \mu=\nu\ \}.$$
For $i>1$, let $$\Lambda_{i} = \{\mu\in\Lambda^+\ , \mu\notin  \Lambda_1\cup\cdots\cup\Lambda_{i-1}\}\, |\ \mbox{ if } \nu\prec\mu\mbox{ for some }
 \nu\in\Lambda^+\mbox{ then } \nu\in \Lambda_1\cup\cdots\cup
 \Lambda_{i-1}\}.$$
The elements of $\Lambda_{i}$ are called \emph{weights of level $i$}.
\end{defn}

\begin{lemma}\label{abcd} The sets $\Lambda_1$ and $\Lambda_2$ for the root systems of types $A_n$, $B_n$, $C_n$ and $D_n$ are given in the table below. In addition, we have\begin{enumerate}[]
\item{\rm {(1)}} for $\Phi=B_n$, $n>2$, $\om_2$ is the only radical weight in $\Lambda_3$;
\item{\rm {(2)}} for $\Phi=C_n$, $n>3$, $2\om_1,\om_4$ are the only radical weights in $\Lambda_3;$
\item{\rm {(3)}} for $\Phi=C_2$ or $C_3$, $2\om_1$ is the only radical weight in $\Lambda_3$.
\end{enumerate}

\begin{table}[h]
$$\begin{array}{|l|c|c|}
\hline
\Phi& \Lambda_1&\Lambda_2 \\
\hline
A_n,n\geq 1&0,\omega_1,\dots,\omega_n& 2\omega_1,2\omega_n,\omega_1+\omega_n,\omega_1+\omega_i,\omega_i+\omega_n, i=2,\dots,n-1\\
\hline
B_n, n\geq3& 0,\omega_n&  \omega_1,\omega_1+\omega_n \\
\hline

    C_n, n>2&0,\omega_1&\omega_2,\omega_3 \\
    \hline
    C_2&0,\omega_1&\omega_2,\omega_1+\omega_2\\
\hline
D_n, n>4&0,\omega_1,\omega_{n-1},\omega_n& \omega_2,\omega_3,\omega_1+\omega_{n-1},\omega_1+\omega_n\\
\hline
D_4&0,\omega_1,\omega_3,\omega_4&\omega_2,\omega_1+\omega_3,\omega_1+\omega_4,\omega_3+\omega_4\\
\hline
\end{array}$$
\end{table}
 \end{lemma}

 \begin{proof} By Lemma~\ref{wtlattice}(3), $\Lambda_1$ consists of minuscule weights and the weight $0$, justifying the entries
   in the column headed $\Lambda_1$ of the above table. Furthermore, $\Lambda_2$ contains a unique radical weight,
   which is the maximal short root (see for instance \cite[Proposition 10]{SZ06}).

Let now $\om=\sum a_i\om_i\in\Lambda_2$ be a non-radical weight.  Suppose that  $a_i\geq 2$ for some $i$.
Then $\om'=\om-\al_i\in\Lambda^+$, so $\omega'\in\Lambda_1$. Inspecting $\Lambda_1$ and the expressions of simple roots in terms of
fundamental dominant weights, we observe that $\om'+\al_i$ (for $\om'\in\Lambda _1$) is dominant
only if $\Phi$ is of type $A_n$ and
$\om\in\{2\om_1,2\om_n\}$; furthermore, it is straightforward to see that in this latter case, we have $2\omega_1, 2\omega_n\in\Lambda_2$.
So we can assume that $a_i\leq 1$ for all $i$. Next we proceed case-by-case, still assuming $\omega\in\Lambda_2$ a non-radical weight.

Consider first the case where  $\Phi=A_n$. If $n=1,2$ then the result is clear, so assume now $n>2$.
Note that $\om_i+\om_j\succ\om_{i-1}+\om_{j+1}$ for $1\leq i<j\leq n$ as  $\om_i+\om_j-\om_{i-1}-\om_{j+1}=
\al_i+\cdots +\al_j$. (Here $\om_0$ and $\om_{n+1}$ are understood to be zero.) So if $a_i,a_j\neq 0$ for some $i\ne j$, then
$\om=\om'+\om_{i-1}+\om_{j+1}$ with $\om'\in\Lambda_1$.  Using the same reasoning for different pairs of non-zero coefficients,
we see that either $i=1$ and $\om'=\om_1$ or $j=n$ and $\om'=\om_n$. Finally,
one observes that no weight obtained is subdominant to another one. So $\Lambda_2$ is as in the table.
This completes the consideration of $\Phi=A_n$.

For  $\Phi\ne A_n$, the argument differs, as some fundamental dominant weights are radical. Recall that $\omega=\sum a_i\omega_i\in\Lambda_2$
is a non-radical weight and we have seen that  $a_i\leq 1$ for all $i$.  If $\om_i$ is a radical weight and $a_i>0$, then $\om-\om_i$ is
subdominant to the weight $\om$, and hence $0\neq \om-\om_i\in \Lambda_1$.
So $\om=\nu+\om_i$, for some $\nu\in\Lambda_1$, $\nu\neq 0$.  Moreover, $\om_i=\mu$, where $\mu$ is the maximal height short root,
as otherwise $\nu+\mu$ is
 subdominant to $\om$ and $\om\notin\Lambda_2$.   So either $\om=\nu+\mu$, for some $\nu\in\Lambda_1$,
 or $a_i=0$ for all $i$ such that $\om_i$ is radical.
 For each root system, we determine when $\nu+\mu$ lies in $\Lambda_2$.

Consider the case $\Phi=B_n$, $n\geq 3$. Following the notation of the previous paragraph, we have
$\nu=\om_n$, $\mu=\om_1$. Moreover, $\omega_i$ is radical for every $i<n$. So $\omega\in\Lambda_2$  non-radical implies
that $\omega=\omega_1+\omega_n$. It is straightforward to verify that $\omega_1+\omega_n\in\Lambda_2$. We deduce that
$\Lambda_2=\{\om_1, \om_1+\om_n\}$. For the final claim of (1), 
let $\om\in\Lambda_3$  be a radical weight. If   $a_i\geq 2$ for some $i$, then $\om-\al_i$  is a radical dominant weight which
must lie in $\Lambda_2$. We deduce that $\omega-\alpha_i = \omega_1$ and we find that $n=2$ and $\omega=2\omega_2$. It is then
straightforward to verify that $2\omega_2\in\Lambda_3$.
So we may now assume $a_i\leq 1$ for all $i$. In particular, as $\omega$ is radical, $a_n=0$. In addition,
$\om_i=\om_{i-1}+\al_i+\cdots +\al_n$, see \cite[Planche II]{Bo},   i.e. $\om_{i-1}\prec\om_{i}$. So $\om\in\Lambda_3$ then implies that $\omega=\omega_2$.

Consider now the case $\Phi=C_n$, for $n\geq2$.  If $a_i\ne 0$ or some $i$ such that $\omega_i$ is radical (as above),
we find that $\nu=\om_1$, $\mu=\om_2$. In this case $\mu+\nu=\om_1+\om_2$. But
$\omega_1+\omega_2-\alpha_1-\alpha_2$ is subdominant to $\omega$ and lies in $\Lambda_1$ only if $n=2$. 
We may now assume $a_i=0$ if $\om_i$ is radical, so $a_i=0$   for $i$ even. Also by the preliminary remarks, $a_i\leq 1$ for all $i$.
It is easy to observe that  $\om_i\succ\om_{i-2}$ for $i>1$, which implies the result on $\Lambda_2$.
We now turn to the final claims of (2) and (3), so let $\om\in\Lambda_3$ be a radical weight.
 If  $a_i\geq 2$ for some $i$, then $\om-\al_i\in  \Lambda_2$ if only if
 $\omega=2\omega_1$.   So we now assume $a_i\leq 1$ for all $i$. Let $1\leq i\leq n$ be maximal such that $a_i=1$.
 Since the dominant weight $\om-\om_i+\om_{i-2}\prec\omega$ must lie in $\Lambda_1\cup \Lambda_2$ and is a radical weight, we
 find that $n\geq 4$ and $\omega=\omega_4$. Finally, one checks that $\omega_4$  lies in $\Lambda_3$.

 Finally consider the case $\Phi=D_n$, $n\geq 4$. Here, in the case where $a_i\ne 0$ for some $i$ with $\omega_i$ radical,
 we have (in the previously defined notation) $\nu\in \{\om_1,\om_{n-1},\om_n\}$ and $\mu=\om_2$, so
  $\mu+\nu\in \{\om_1+\om_2,\om_2+\om_{n-1},\om_2+\om_n\}$.  Now
  $\om_2+\om_n\succ\om_1+\om_{n-1}\notin\Lambda_1$ and $\om_2+\om_{n-1}\succ\om_1+\om_{n}\notin\Lambda_1$
  so $ \om_2+\om_n,\om_2+\om_{n-1}\not\in\Lambda_2$. Furthermore,
as $\om_1+\om_2-\al_1-\al_2=\om_3+\delta_{n,4}\omega_4$,
it follows that $\om_1+\om_2\not\in\Lambda_2$. So we now assume that $a_i=0$ for all $i$ such that $\omega_i$ is radical,
that is, $a_i=0$ if $i<n-1$ is even and as established earlier $a_j\leq 1$ for all $j$. Moreover, there are at most two $a_j$ which are non-zero, as otherwise there exists $\beta\in\Phi$ with $\omega-\beta$ dominant and not lying in $\Lambda_1$.  
Suppose $a_i=1$ for some (odd)
$i<n-1$.
Then $\omega-(\omega_i-\omega_{i-2})\prec\omega$ must lie in $\Lambda_1$ and so $i=3$.
 So finally, recalling that $\omega$ is non-radical we have
$\om\in\{\omega_3 (n>4),\omega_3+\omega_{n-1} (n>4), \omega_3+\omega_n (n>4),\om_1+\om_n, \om_1+\om_{n-1}, \om_{n-1}+\om_n\}$.
It is straightforward to see that $\omega_3 (n>4), \om_1+\om_{n-1}$ and $\om_1+\om_n$ all lie in $\Lambda_2$.
In addition, $\om_{n-1}+\om_n\succ \omega_{n-3}$, and the latter lies in $\Lambda_1$ if and only if $n=4$.  So
it remains to show that  $\omega_3+\omega_n,\omega_3+\omega_{n-1}\notin\Lambda_2$ for $n>4$. This is clear since
$\omega_2+\omega_{n-1}$, respectively $\omega_2+\omega_n$, is  subdominant to the given weight and does not lie in $\Lambda_1$.
\end{proof}

Now combining Lemma~\ref{abcd} with Lemma~\ref{wtlattice}, we obtain some information about the weight lattice of certain irreducible  $G$-modules.

\begin{lemma} \label{mr1} Assume $p=0$ or $p>e(G)$. Let $\om\neq 0$ be a  $p$-restricted dominant weight for $G$.  If $\om\not\in\Lambda_1\cup\cdots\cup\Lambda_i$ for
    some $i>0$, then there are   weights
    $\nu_1\ld \nu_i$ of $V_\om$ such that $\nu_j\in\Lambda_j$ for $j=1\ld i$. In addition, the weights of $V_{\nu_j}$ occur
    as weights of $V_{\om}$, for $1\leq j\leq i$.

\end{lemma}

\begin{proof} This follows from the definition of $\Lambda_j$ and Lemma~\ref{wtlattice}.  \end{proof}


\section{Proof of Theorem~\ref{c99}}

In this section, we prove Theorem~\ref{c99}, so in particular we are concerned with the action of non-central non-regular semisimple elements on certain specific representations (as shown by Theorem~\ref{ag8}). As noted earlier, in remark~\ref{rm5}(3), we must handle some small rank groups as well as the groups $B_n$, $D_n$, and $C_n$ when $p=2$ and for certain highest weights; we do this in two separate subsections.

\subsection{Groups of small rank}\label{sec:smrk}

\begin{lemma} \label{sl3} Let $G=A_2$ and let $s\in T\setminus Z(G)$ be a non-regular element.
Let $V=V_\om$ be the \ir  $G$-module of $p$-restricted highest weight $\om\neq 0$.
Then the spectrum of $s$ on $V_\om$ is almost simple  if and only if $\om = \omega_1$ or $\omega_2$.
\end{lemma}

\begin{proof} We take $T$ to be the torus of diagonal matrices in ${\rm SL}_3(F)$. Since $s$ is non-regular non-central, with respect to an
  appropriate choice of base of $V_{\omega_1}$, we may assume $s= \diag(a,a,a^{-2})$, for some $a\in F$ with $a^3\ne 1$. Clearly the
  spectrum of  $s$ on $V_{\omega_1}$ and $V_{\omega_2}$ is indeed almost simple. So we now assume $\om\notin\{0,\om_1,\om_2\}$.
  In particular, Lemma~\ref{abcd} implies  $\om\notin\Lambda_1$ and by Lemmas~\ref{mr1} and ~\ref{abcd},  $\Omega(V_\om)$ has
  some weight from $\Lambda_2=\{\om_1+\om_2,2\om_1,2\om_2\}$.

Suppose first that $\om=2\om_1$, and so $p\ne 2$.  The weights of $V_{\om_1}$ are $\{\ep_1,\ep_2,\ep_3\}$, so the
weights of $V_{\om_1}\otimes V_{\om_1}$ are $2\ep_1,2\ep_2,2\ep_3, \ep_1+\ep_2,\ep_1+\ep_3,\ep_2+\ep_3$, which by
Lemma~\ref{wtlattice}(2) coincide with the weights of $V_\om$. Now, $2\ep_1(s)=2\ep_2(s)=a^2$, and $(\ep_1+\ep_3)(s)=(\ep_2+\ep_3)(s)=a\up$.
As $a^3\neq 1$, the \eis $a^2,a\up$ are distinct, so the spectrum of   $s$ on $V_{2\om_1}$ is not almost simple,
as claimed. Since $V_{2\omega_2} $ is dual to $ V_{2\omega_1}$, the spectrum of $s$  on $V_{2\omega_2}$ is not almost simple as well.

Suppose  now that $\om=\om_1+\om_2$. Then the weights of $V_\om$ are the roots and the zero weight.
Then $(\al_1+\al_2)(s)=\al_2(s)=(\ep_2-\ep_3)(s)=a^3\neq 1$ and $-(\al_1+\al_2)(s)=-\al_2(g)=a^{-3}$.
If $p\neq 3$, the eigenvalue 1 is also of multiplicity $2$,
and we are done. If $p=3$ and $a^3\neq a^{-3}$ then we are done as well. So suppose $p=3$ and $a^6=1$ and hence $a^3=-1$, that is $a=-1$. Note that  $\pm \al_2(s)=-1$ and $\pm\al_1(s)=1$, so the result also follows in this case.

We now appeal to Lemma ~\ref{mr1} to conclude. \end{proof}

\def\so{the spectrum of }

\begin{lemma}\label{ht2} Let $G=A_3$ and let $s\in T\setminus Z(G)$ be a  non-regular  element.
Let $V=V_\om$ be the \ir  $G$-module of $p$-restricted highest weight $\om\neq0$.
If \so $s$ on $V$ is almost simple, then
either $\om=\om_1$ or $\om_3$, or $\om= \om_2$ and there exists $a\in F$, $a^4\ne 1$ such that
with respect to a suitably chosen basis,  $s=\diag(a,a,\pm a\up,\pm a\up)$, $a\in F$.\end{lemma}

\begin{proof}  Without loss of generality, we take $T$ to be the set of diagonal matrices in ${\rm SL}_4(F)$. We may assume $s=\diag(a,a,b,c)$ for some $a,b,c\in F$ such that $a^2bc=1$. Fix the base of $\Phi$ such that
  $\alpha_i(\diag(a_1,a_2,a_3,a_4)) = a_ia_{i+1}^{-1}$ for $1\leq i\leq3$; in particular $\al_1(s)=1$.
It is clear that if $a^2b^2\ne 1$ then $s$ has almost simple spectrum on
$V_{\om_1}$ and on $V_{\om_3}$. 
If $\om=\om_2$, then the matrix of $s$ on
$V$ is conjugate to $s_1=\diag(a^2,a^{-2},ab,ab,(ab)^{-1},(ab)^{-1})$,   so \so $s$ is almost simple
only if $b=\pm a^{-1}$ and $a^4\ne 1$, and the result easily follows.

Now consider the general case, where $\omega\not\in\{\om_1,\om_2,\om_3\}$. Assume $s$ has almost simple spectrum on $V_\om$.
Factor $s$ as $$s = \diag(a\gamma,a\gamma,a^{-2}\gamma^{-2},1)\cdot\diag(\gamma^{-1},\gamma^{-1},\gamma^{-1},c),$$
where $\gamma,c\in F$ with $\gamma^3=c$.
Then viewing $s$ as lying in the maximal parabolic $P=LQ$, $Q=R_u(P)$,
corresponding to the root $\alpha_3$, we see that the second factor acts as a scalar on the
fixed point space $V_\om^Q$. Hence the eigenvalue multiplicities of $s$ on this fixed point
space are determined by those of the first factor. We now apply Lemma~\ref{sl3} to the element
$h=\diag(\gamma a,\gamma a,(\gamma a)^{-2})$ and the weight $\omega\downarrow L'$, which is the highest weight of the
irreducible $L'$-module $V_\om^Q$. In addition, we apply Lemma~\ref{sl3} to  $(V_\om^*)^Q$. 
By  Lemma~\ref{sl3}, the only $p$-restricted irreducible representations of  $\SL_3(F)$ on which $h$ has an almost simple
spectrum are the natural representation and its dual.  Writing $\om=m_1\om_1+m_2\om_2+m_3\om_3$, we deduce  that
$(m_1,m_2),(m_2,m_3)\in \{(0,0),(1,0), (0,1)\}$.
We are therefore reduced to considering the case
$\om=\om_1+\om_3$, (a quotient of) the adjoint representation. The \mult of the weight $0$ is at least $2$ and $\al_1(s)=1$.
Therefore, $(\al_1+\al_2)(s)=\al_2(s)$, so $\al_2(s)=1$ as well.
But then $(\al_2+\al_3)(s)=\al_3(s)\ne 1$, as $s$ is non-central; hence $s$ is not almost cyclic on $V_{\om_1+\om_3}$.\end{proof}

 \begin{lemma}\label{C22C} Let $G=C_2$, 
 $p=2$, and let $\om$  be a non-zero $2$-restricted dominant weight.
Let $s\in T$ be a  non-regular element. Suppose that \so $s$
on $V_\om$ is almost simple. Then $\om\in \{\om_1,\om_2\}$ and \so $s$ is almost simple on precisely one of the modules
$V_{\om_1}$ and $V_{\om_2}$. Assume moreover that $T$ is the torus of diagonal matrices in the group ${\rm Sp}_4(F)$, written
with respect to a fixed symplectic basis $(e_1,e_2,f_2,f_1)$ of the natural module $V_{\omega_1}$. Let $g\in T$ be non-regular.
If the spectrum of $g$ on $V_{\om_1}$ is almost simple then, up to conjugacy, $\ep_1(g)=a$,
$\ep_2(g)=1$ for $1\neq a\in F$;   if \so $g$  on $V_{\om_2}$ is almost simple then, up to conjugacy
$\ep_1(g)=\ep_2(g)=a$ for $1\neq a\in F$. \end{lemma}

\begin{proof} As $\om$ is $2$-restricted, if $\om\not\in\{\om_1,\om_2\}$ then $\om=\om_1+\om_2$, and \cite[\S12, Corollary of Theorem 41]{St} implies that
  $V_{\om}=V_{\om_1}\otimes V_{\om_2}$. By Lemma~\ref{td2}, \so $s$ is simple on $V_{\om_1}$, and hence $s$ is regular,
  contradicting our hypothesis. One easily verifies the validity of the additional assertions. \end{proof}

\begin{lemma}\label{cc2} Let
$G=C_2$,
$p\ne 2$, and fix an ordered symplectic basis  $(e_1,e_2,f_2,f_1)$ of the natural module of $G$ and let $T$ be the torus of diagonal matrices of $G$ in the natural representation. Let $s\in T\setminus Z(G)$ be a
non-regular element and let  $V_\om\in\Irr (G)$ be
non-trivial  $p$-restricted $G$-module. 
Then  $s$ has almost simple spectrum  on  $V_\om$ if and only if one of the following holds:
\begin{enumerate}[]
\item{\rm {(i)}} $\om=\om_1$, and up to conjugacy, $\ep_1(s)=1$, $\ep_2(s)=a$
   or $\ep_1(s)=-1$, $\ep_2(s)=a$, where  $a\in F$, $a^2\neq 1;$
\item{\rm {(ii)}} $\om=\om_2$, and up to conjugacy, $\ep_1(s)=1$, $\ep_2(s)=-1$
   or $\ep_1(s)= \ep_2(s)=a$, where  $a\in F$, $a^2\neq \pm1.$
\end{enumerate}
\end{lemma}

\begin{proof}  Let
  $\ep_1(s)=b$, $\ep_2(s)=a$, that is $s = {\rm diag}(b, a, a^{-1}, b^{-1})$.

  We first consider $\om=\om_1$, so $\Om(V_\om)=\{\pm\ep_1,\pm\ep_2\}$. Since $s$ is non-regular, we may assume that either
  $a=b$ or $b^2=1$. In the first case, $s$ does not have almost simple spectrum on $V_\om$, while in the second case $s$ has
  almost simple spectrum on $V_\om$ if and only if $a^2\ne 1$.

We now turn to the cases  $\om\neq \om_1$.  By Remark~\ref{rm5}(1),
we are left with the exceptional cases described in Lemma~\ref{a12}, 
$s_1=\diag(a,a,a\up,a\up)$ with $a^2\neq1$,  or $s_2=\pm\diag(1,-1,-1,1)$. Note that  $\al_1(s_1)=1$ and $\al_2(s_2)=1$.
By  Lemma~\ref{abcd},    $\Lambda_1 =\{0,\om_1\}$, and $\Lambda_2=\{\om_1+\om_2,\om_2\}$ and $2\om_1$ is the
only radical weight in $\Lambda_3$.   We consider these weights in turn, before turning to the general case.

The weights of $V_{\om_2}$
are $0,\pm\ep_1\pm\ep_2$. The remarks of the preceding paragraph imply that the cases in the statement are the only possible ones,
 and they yield the matrices of $s_1,s_2$ on  $V_{\om_2}$ (with respect to a suitable basis)
  $\diag(a^2,1,1,1,a^{-2})$ and $\diag(-1,-1,1,-1,-1)$, respectively.

Suppose $\om= \om_1+\om_2$. Then  $\Om(V_\om)=\Om(V_{\om_1}\otimes V_{\om_2})$, by Lemma~\ref{wtlattice}.
In terms of Bourbaki weights, the weights in $\Omega(V_\om)$ are $\pm\ep_1 +(\pm\ep_1\pm\ep_2)$, $\pm\ep_2+(\pm\ep_1\pm\ep_2)$,
 $\pm\ep_1$, and $\pm\ep_2$. Then $(\pm\ep_1 +(\pm\ep_1\pm\ep_2))(s_2)=-1$,
$(\pm\ep_2+(\pm\ep_1\pm\ep_2)(s_2)=1$, so \so $s_2$ on $V_\om$ is not almost simple. Furthermore,
$(\ep_1+(-\ep_1+\ep_2))(s_1)=a=(\ep_2+(\ep_1-\ep_2))(s_1)$ and
$(-\ep_1+(\ep_1-\ep_2))(s_1)=a\up=(-\ep_2+(-\ep_1+\ep_2))(s_1)$. So \so $s_1$ on $V_\om$ is not almost simple.

Finally, suppose $\om=2 \om_1$. Then by Lemma~\ref{wtlattice}, the weights of $V_\om$ are the same as those of $V_{\om_1}\otimes V_{\om_1}$.
These are  $\pm\ep_i\pm\ep_j$, for $i,j\in\{1,2\}$. But now it is easy to see that neither $s_1$ nor $s_2$ has almost simple spectrum on $V_\om$.

We now turn to the general case and suppose that $\om$ differs from the weights examined above. Then
$\om\notin  \Lambda_1\cup \Lambda_2$ and $\om\neq 2\om_1$. Recall that if $\mu\in\Lambda_i$ for some $i$ then $V_\mu$ has a weight
from $\Lambda_j$ for every $j=1\ld i-1$ (Lemma~\ref{mr1}). Then Lemma~\ref{abcd} implies that either $2\om_1$ or $\om_1+\om_2$
is a weight of $V_\om$ and by Lemma~\ref{wtlattice}, the weights of $V_{2\om_1}$ or $V_{\om_1+\om_2}$ are weights of $V_\om$.
The above considerations of $V_{\om_1+\om_2}$ and $V_{2\om_1}$ show then that, given $s=s_1$ or $s_2$,  there are $4$ distinct
weights
$\lam_1,\lam_2,\nu_1,\nu_2$ in $\Omega(V_\om)$ such that $\lam_1(s)=\lam_2(s)\neq \nu_1(s)=\nu_2(s)$. So $s$ is not almost cyclic on
 $V_\om$, which completes the proof of the result. \end{proof}

 \begin{lemma}\label{gg2} Theorem~$\ref{c99}$ is true for G of type $G_2$.
 \end{lemma}

 \begin{proof} Let $V$ be a non-trivial $G$-module and $1\neq s\in T$ a non-regular element. We have to show that the
   spectrum of $s$ on $V$ is not almost simple.   Let $\om$ be the highest weight of $V$. Suppose first that $\om=\om_1$ or 
$p=3,\om=\om_2$,
so $\dim V=7$, or 6 for $p=2$.
The group $G$ contains a maximal rank closed subgroup $H$ isomorphic to $A_2$
such that the restriction of $V_{\om_1}$ to $H$ is completely reducible; the irreducible constituents are the natural module for 
$\SL_3(F)$, and its dual and, if $p\neq 2$, an additional trivial summand. So the matrix of $s$ on $V_{\om_1}$ can be written as 
$\diag(a,b,c,1,a\up,b\up,c\up)$ if $p\neq 2$, otherwise $\diag(a,b,c,a\up,b\up,c\up)$, where $abc=1$ in both cases. This is 
also true if $p=3$ and $V=V_{\om_2}$. If all the entries are distinct, this matrix is a regular element in $\SL(V)$, and hence 
in $G$, contrary to the assumption.

Suppose that the entries are not distinct. As any permutation of $a,b,c$ can be realized by an inner automorphism of  $G$, we
 may assume that $a$ equals some other diagonal entry and by the same reasoning, we may ignore the possibilities $a=c$ 
and $a=c\up$. So we examine the cases  $a=b$, $a=a\up$, and $a=b\up$.

 Let $a=b$.  Then $s$ has almost simple spectrum on $V_\om$ only if $a=a\up$. But then $c=1$ and $s$ is not almost cyclic on $V_\om$.   

Let $a=a\up \neq b$, so  $a=\pm 1$, $c=\pm b\up$. If $a=1$, then $b\neq 1$,
$s$ acts on $V_\om$ as $\hat s=\diag (1,b,b\up,1,1,b\up, b )$  (where we drop the 1 in the middle if $p=2$)  which does not have
almost simple spectrum. If $a=-1$ then 
$p\neq 2$ and  $s$ acts on $V_\om$ as $\hat s=\diag (-1,b,-b\up,1,-1, b\up, -b )$. If $b=\pm 1$ then \so $\hat s$ is not almost simple.
Let $b\neq \pm 1$.
As $V$ is an orthogonal space, $s$ is a regular element of $SO(V)$ (Lemma~\ref{co1}), and hence in $G$, contrary to the assumption.

Let $a=b\up$. Then $c=1$.  By reordering $a,c$, we arrive at the case $a=1$, considered above.
 This completes the analysis of the cases $\om=\om_1$, and $(\om,p) = (\om_2,3)$.

 Suppose now that $\om$ is an arbitrary $p$-restricted weight. If $p\neq 2,3$ then the weights of  $V_{\om_1}$ occur as 
weights of $V$ (Lemma~\ref{wtlattice}),
 so the result follows from that for $V_{\om_1}$. Let $p=2$; now $0,\om_1,\om_2,\om_1+\om_2$ are the only $2$-restricted dominant weights of $G$.  By \cite[Theorem 15]{Z}, the
weights   of $V_\om$  are the same as in characteristic $0$, in particular all weights of $V_{\om_1}$ are weights of $V_\om$, and
we conclude as above.

Now turn to the case $p=3$ and $\omega$ still $p$-restricted. By \cite[Theorem 15]{Z}, if
 $\om\neq 2\om_2$ then the weights of $V_\om$ are the same as in characteristic $0$, and in particular
all weights of $V_{\om_1}$ are weights of $V_\om$. So the result follows as above. For $p=3$ and $\om=2\om_2$, we use the tables of 
\cite{Lu1} to see that the weights of $V_{\om_2}$ are weights of $V_{2\om_2}$, and then conclude as before.

Finally, suppose that  $\om$ is not $p$-restricted.  By Remark~\ref{nonrest}, we may assume that   $V$ is tensor-decomposable, say,
$V=V_1\otimes V_2$, where the highest weight of $V_1$ is of the form $p^k\om'$ for some $k$.
Then the result follows by Lemma~\ref{td2}.\end{proof}

\subsection{Groups $B_n$ with $n>2$,  $D_n$ with $n>3$, and $C_n$ with $p=2$ and $n>2$}\label{BCD}
$\newline$

In this section, we consider the groups as indicated in the heading of the
section. Recall that when $G=B_n$, we may assume $p\ne 2$. Note that for groups $G$ of type $B_n$ and of type $D_n$,
the \mult of the $0$ weight in the adjoint
\rep $V_{\om_2}$ is greater than $1$.  Therefore, if $\om$ is a dominant weight such that
$\om_2\prec\om$ then,  by Lemma~\ref{no1}(2), it suffices to observe that a non-central
non-regular semisimple element $s\in G$ is not of almost simple spectrum  on $V_{\om_2}$. This is done in Lemma~\ref{bd0} below.
The condition $\om_2\prec\om$ holds provided $\om$ is a  radical weight and $\om\neq 0,\om_1,\om_2$ for $G$ of type $B_n$, and
$\om\neq 0,\om_2$ for $G$ of type $D_n$.

\begin{lemma}\label{bd0} Let $G=B_n$, $n>2$, $p\ne 2$, $\omega\in\{\omega_2,\omega_n\}$ or $G=D_n,n>3$,
  $\omega\in\{\omega_2,\omega_{n-1},\omega_n\}$. Let $s\in T\setminus Z(G)$ be a non-regular
 element. Then \so $s$ on $V_\omega$ is not almost simple, unless $G=D_4$, $\omega\in\{\omega_3, \omega_4\}$.\el

 \begin{proof} Here we take $T$ to be the preimage in $G$ of the set of diagonal matrices in the image of $G$ under the natural representation.
We take $s\in T$ and assume \so $s$  on $V_\omega$ is almost simple. Since $s$ is not regular, there exists a root
   $\alpha$ with respect to $T$ such that $\alpha(s)=1$. We will assume without loss of generality that either $\alpha = \alpha_1$,
   or $G=B_n$ and $\alpha=\alpha_n$. 

Suppose first that $\om=\om_2$. Set $R_0 = \{\alpha\in\Phi\ |\ \alpha(s)=1\}$. Since $s$ is non-central, there exists
$\beta\in\Phi\setminus R_0$. Moreover, since $\Phi$ is an irreducible root system, there exists $\beta\in\Phi\setminus R_0$
which is not orthogonal to $R_0$. So for some $\alpha\in R_0$, $w_\beta(\alpha)\ne \beta$. Let $w_\al\in W(G)$ be the reflection
corresponding to $\al.$ Then $\beta(s) = w_\alpha(\beta)(s)\ne 1$, while $\alpha(s)=-\alpha(s) = 1$. So $s$ is not almost
cyclic on $\omega_2$.

Let $\om\in\{\om_{n-1},\om_n\}$, for $G=D_n$ and $n>4$,  or $\om=\om_n$ for $G=B_n$. Then
$\mu=\frac{1}{2}( \al_1+\nu)$ is a weight of $V_{\om}$, for $\nu\in\{\pm\ep_3\pm \cdots \pm\ep_{n}\}$,
with certain condition on the parity of the number of  minus signs in the $D_n$-case. Suppose that $\al=\al_1$.
Then $\mu-\al_1$ is a weight of $V_\om$ for any admissible choice of the signs. As \so $s$  on $V_\om$ is almost simple, we deduce that $\mu(s)$ does not depend on the choice of $\nu$ and so $\ep_3(s)=\cdots =\ep_n(s)=1$. Similarly, this then implies that $(\frac{1}{2}( \ep_1+\ep_2+\nu))(s)$ does not
depend on the choice of $\nu$, so again this value must be equal to   $(\frac{1}{2}( \pm(\ep_1-\ep_2)+\nu))(s)$,
whence $\ep_1(s)=1=\ep_2(s)$ as well. This implies $s\in Z(G)$, a contradiction.

Finally, suppose that $G=B_n$, $\omega = \omega_n$ and $\al=\alpha_n$.
Then for all $1\leq i\leq n-1$, we have the two distinct weights of $V_\omega$, $\omega-\alpha_i-\alpha_{i+1}-\cdots-\alpha_{n-1}-\alpha_n$ and 
$\omega-\alpha_i-\alpha_{i+1}-\cdots-\alpha_{n-1}-2\alpha_n$, taking the same value on $s$,  and therefore deduce that $\alpha_i(s) = 1$ for all $i$,
again contradicting the fact that $s$ is non-central.\end{proof}

\begin{remk} If $G=D_4$  then  there exist non-central non-regular semisimple elements $s$ with almost simple spectrum on  $V_ {\omega_3}$ or
  $V_{\omega_4}$. Indeed, one easily
  observes that there are  non-regular elements $s\in T\setminus Z(G)$
  whose spectrum is almost simple on $V_ {\omega_1}$. Let $\si$ be  the triality automorphism of $G$. Then $\si(s)$ has almost
  simple spectrum  on $V_{\om_1}^\sigma$, whence the claim.\end{remk}

\begin{lemma}\label{Bn2} Let $G=B_n$, $n>2$, $p\ne 2$, or $G=D_n$, $n>4$ , and let $V_\om\in\Irr(G)$, where
$\om\neq 0$ is p-restricted. Let $s\in T\setminus Z(G)$ be a non-regular
 element with almost simple spectrum  on $V_\om$. Then  $\om= \om_1$. If $G=D_4$ then $\om\in \{\om_1,\om_3, \om_4\}$. \end{lemma}

\begin{proof} If $\om$ is radical, this follows from Lemmas~\ref{bd0}, ~\ref{no1}(2) and ~\ref{wtlattice}, both for $B_n$
and $D_n$.

Suppose that $\om$ is not radical. If $G=B_n$ then $\om_n\preceq \om$ by Lemma~\ref{wtlattice}(2),
so again the result follows from Lemmas~\ref{bd0} and ~\ref{no1}(1)(i). Let $G=D_n$.
By Theorem~\ref{ac4}, 
all non-zero weights of $V_\om$
are of \mult 1. Then, by \cite[Tables 1,2]{TZ2}, $\om\in\{\om_1,2\om_1,\om_2,\om_{n-1},\om_n\}$, where the radical weights
$2\om_1,\om_2$ are to be dropped. Whence the result for $n=4$. If $n>4$ then \so $s$ on $V_{\om}$ is not almost simple  by
Lemma~\ref{bd0}. \end{proof}

We now handle the case  $G=C_n$, for $n>2$ and $p=2$, which is excluded in Proposition~\ref{vg1}. Moreover,
we only need to consider $V_{\om_n}$ (see Proposition~\ref{vg1}).

\begin{lemma}\label{Cnexc} Let $G=C_n$, $n>2$, $p=2$.
Let $1\neq s\in T$ be a  non-regular  element. Then \so $s$ on $V_{\omega_n}$ is not almost simple.
\end{lemma}

\begin{proof} We argue as in the proof of Lemma~\ref{bd0}. We can assume that $\al(s)=1$ for $\al=\al_1$ or $\alpha=2\ep_1$.
The weights of $V_{\om_n}$ are $\pm\ep_1\pm\cdots \pm\ep_n$. Then  $\mu:=\ep_1\pm\ep_2+\nu$ are weights of $V_{\om_n}$
for any $\nu=\pm\ep_3\pm\cdots\pm\ep_n$.
If $\al=2\ep_1$ then $\mu-\al$ is a weight of $V_{\om_n}$, and we conclude (as in the proof of Lemma~\ref{bd0}) that
$(2\ep_1)(s)=\cdots =(2\ep_n)(s)$. As $p=2$, we have $\ep_1(s)=\cdots =\ep_1(s)$, whence $s\in Z(G)=1$.

If $\al=\al_1$ then for $\mu=\ep_1-\ep_2+\nu$ we have $\mu-\al_1\in\Omega(V_{\om_n})$, whence
$(2\ep_3)(s)=\cdots =(2\ep_n)(s)=1$. This implies that $(\ep_1+\ep_2+\nu)(s)$ does not depend on $\nu$, nor does $(\epsilon_1-\epsilon_2+\nu)(s)$,
whence $(2\ep_1)(s)=1$, and again we conclude that $s=1$.\end{proof}

\subsection{Proof of Theorem~{\rm\ref{c99}}}

\begin{proof} Let $G,s$ be as in the statement of Theorem~{\rm\ref{c99}}.
Note that ${\rm rank}(G)\geq 2$.

Suppose first that $\lam$ is $p$-restricted.
The groups of rank $2$ have been examined in Lemmas~\ref{sl3}, ~\ref{C22C}, ~\ref{cc2} and ~\ref{gg2}, and
the group of type $A_3$ in Lemma~\ref{ht2}. In Proposition~\ref{vg1}, we handled the groups
$A_n,n>3$,  $F_4$, $E_6,E_7,E_8$, and all $p$-restricted weights for the group $C_n$, $n>2$, except the weight
$\omega=\omega_n$ when $p=2$. The latter is handled in Lemma~\ref{Cnexc}.

Groups of type $B_n,n>2,p\ne 2$ and  $D_n$ are dealt with in Lemma~\ref{Bn2}. 

By remark~\ref{nonrest}, we may now assume that $V$
is tensor-decomposable. Let $V=V_1\otimes V_2$ be a non-trivial tensor decomposition of $V$.   By Lemma~\ref{td2}, the
spectra of $s$ on  $V_1$ and $V_2$ are simple. This contradicts Lemma~\ref{td2}.   \end{proof}

Finally, we conclude with a straightforward corollary of Theorem~\ref{c99}.

\begin{corol}\label{em5} Let $s\in T\setminus Z(G)$ be a  non-regular element and V an \ir $G$-module.
Suppose that \so s  on  V is almost simple. Then
the \ei multiplicities of $s$ on V do not exceed $m=m_V(s)$, where  either $m\leq {\rm rank}(G)$ or
one of the \f holds:

$(1)$ $G=A_3$, $\dim V=6$, $m =4;$

$(2)$ $G=B_n$, $n>2$, $p\neq 2$,  $\dim V=2n+1$, $m=2n$;

$(3)$ $G=C_n$, and either $\dim V = 2n$ and $m=2n-2$ or $n=2$, $p\ne 2$, $\dim V = 5$ and $m=4$;

$(4)$ $G=D_n$, $n>3$, $\dim V=2n$, $m=2n-2$.
\end{corol}

\begin{proof} This will follow from Theorem~\ref{c99}; we discuss each of the cases of the theorem. To get (1)
  above, we additionally use Lemma~\ref{ht2}.
For $G=C_2$, $p\neq 2$, we use Lemma~\ref{cc2}. The modules
of dimensions indicated in Theorem~\ref{c99}(1) are obtained by Frobenius twisting of $V_{\om_1}$ (where the statement is clear); $m_V(s)$ remains unchanged under such a twist.
This leaves us with $G=D_4$ and $\dim V = 8$. The modules $V_{\om_1},$  $V_{\om_3},$ $V_{\om_4},$ are obtained from each other by a graph automorphism of $G$, and the other modules of dimension $8$ as in  Theorem~\ref{c99}(4) are Frobenius twists of these. The result follows.\end{proof}

\end{document}